\documentclass{imanum}
\usepackage{graphicx}

\jno{drnxxx}

\begin{document}

\title{WSLD operators \uppercase\expandafter{\romannumeral2}: the new fourth order difference approximations for space Riemann-Liouville derivative}

\author{%
{\sc
Minghua Chen
and
Weihua Deng\thanks{Corresponding author. Email: dengwh@lzu.edu.cn}} \\[2pt]
School of Mathematics and Statistics,
Lanzhou University, Lanzhou 730000, P. R. China
}

\maketitle

\begin{abstract}
{ High order discretization schemes play more important role in fractional operators than classical ones. This is because usually for classical derivatives the stencil for high order discretization schemes is wider than low order ones; but for fractional operators the stencils for high order schemes and low order ones are the same. Then using high order schemes to solve fractional equations leads to almost the same computational cost with first order schemes but the accuracy is greatly improved. Using the  fractional linear multistep methods, Lubich
 obtains the $\nu$-th order ($\nu\leq 6$)  approximations of the $\alpha$-th derivative ($\alpha>0$) or integral ($\alpha<0$) [Lubich, SIAM J. Math. Anal., 17, 704-719, 1986], because of the stability issue the obtained scheme can not be directly applied to the space fractional operator with $\alpha \in(1,2)$ for time dependent problem. By weighting and shifting Lubich's 2nd order discretization scheme, in [Chen \& Deng, arXiv:1304.7425] we derive a series of effective high order discretizations for space fractional derivative, called WSLD opeartors there. As the sequel of the previous work, we further provide new high order schemes for space fractional derivatives by weighting and shifting Lubich's 3rd and 4th order discretizations.  In particular, we prove that the obtained 4th order approximations are effective for space fractional derivatives. And the corresponding schemes are used to solve the space fractional diffusion equation with variable coefficients.
}
{ Fractional derivatives; High order scheme; Weighted and shifted Lubich difference operators;  Numerical stability
}
\end{abstract}

\section{Introduction}
\label{sec;introduction}
Fractional calculus (i.e., integrals and derivatives
of any arbitrary real or even complex order) has attracted considerable attention during the past several decades, due mainly to its demonstrated
applications in seemingly diverse and widespread fields of science and
engineering [\cite{Kilbas:06}]; and fractional derivatives provide an excellent tool
for the description of memory and hereditary properties of various materials and processes [\cite{Podlubny:99}]. With the ubiquitous applications of fractional calculus, fractional partial differential equations (PDEs) appear naturally. Effectively solving fractional PDEs becomes urgent, and intrigues mathematicians. It is still possible to analytically solve the linear fractional PDEs with constant coefficients by using Laplace or Fourier transform, but most of the time the solutions are represented by infinite series or transactional functions.  Without doubt, the new challenges also exist in numerically solving fractional PDEs; but some basic ideas have been developed, for instance,  finite difference method [\cite{Meerschaert:04,Sousa:12,Sun:06,Tian:12,Zhuang:09}]; finite element method [\cite{Ervin:06,Deng:2008}]; spectral method [\cite{Li:09}].

In numerically solving fractional PDEs, besides a little bit complex numerical analysis, the big challenge comes from the computational cost caused by the nonlocal properties of fractional operators. High order scheme is a natural idea to reduce the challenge of cost. Comparing with first order schemes, the high order schemes for fractional operators do not increase computational cost but greatly improve the accuracy. The reason is that both the derived matrixes corresponding to the higher order schemes and low order schemes are full and have the same structure [\cite{Deng:13}]. In fact, there are already some important progresses for the high order discretizations of fractional derivatives, including WSGD opeartor [\cite{Tian:12}], CWSGD opeartor [\cite{Zhou:13}], second order discretization [\cite{Sousa:12}], second order discretization for Riesz fractional derivative [\cite{Ortigueira:06}], and WSLD operator [\cite{Chen:13}].  This paper is the sequel of [\cite{Chen:13}], i.e,  based on Lubich's 3rd and 4th operators to provide new high order discretization schemes for space fractional derivatives.

 Using the  fractional linear multistep methods,
\cite{Lubich:86} obtains the $\nu$-th order ($\nu\leq 6$)  approximations of the $\alpha$-th derivative ($\alpha>0$) or integral ($\alpha<0$) by the corresponding coefficients of the generating functions $\delta ^{\alpha}(\zeta)$, where
\begin{equation}\label{1.1}
\delta^{\alpha}(\zeta) = \left(\sum_{i=1}^\nu\frac{1}{i}(1-\zeta)^i\right)^{\alpha}.
\end{equation}
 For $\alpha=1$,  the scheme reduces to the classical $(\nu+1)$-point backward difference formula  [\cite{Henrici:62}]. For $\nu=1$, $\alpha>0$, the scheme (\ref{1.1}) corresponds to the standard Gr\"{u}nwald discretization of $\alpha$-th derivative with first order accuracy;
unfortunately, for the time dependent equations all the difference discretizations are unstable. By weighting and shifting Lubich's 2nd order discretization, a class of effective high order schemes for space fractional derivatives are presented [\cite{Chen:13}]. Is it possible to design the high order schemes for space fractional derivatives by using Lubich's 3rd, 4th, 5th, 6th order operators?  This paper will answer that at least by applying Lubich's 3rd, 4th order operators, the new discretizations for space fractional derivatives can be constructed. The concrete discretizations will be presented, and the effectiveness of 4th order schemes for space fractional derivative will be proved. And we will also provide a simple application to solve the space fractional diffusion equation with variable coefficients.

The  outline of this paper is as follows.
In Section 2, we derive the new fourth order approximations for space fractional Riemann-Liouville derivatives, being
 effective in solving space fractional PDEs. A simple application of the new discretization schemes are presented in Section 3 to solve the space fractional diffusion equation with variable coefficients. And in Section 4, the numerical experiments are performed to show the effectiveness of the algorithm and verify the theoretical results. Finally, the paper is concluded with some remarks in the last section.

\section{Derivation of new fourth order discretizations for space fractional operators}\label{sec:1}
Based on Lubich's 3rd and 4th discretizations, we derive new fourth order approximations for Riemann-Liouville derivative, and prove that they are effective in solving space fractional PDEs, i.e., all the eigenvalues of the matrixes corresponding to the discretized operators have negative real parts.


\begin{definition}[\cite{Podlubny:99}]\label{definition2.7}
The $\alpha$-th ($n-1<\alpha<n$) order left and right Riemann-Liouville  fractional
derivatives of the function $u(x)$ on $[x_L,x_R]$, $-\infty \leq x_L <x_R \leq \infty$ are, respectively, defined by
\begin{equation*}
 _{x_L}D_x^{\alpha}u(x)=
\frac{1}{\Gamma(n-\alpha)} \displaystyle \frac{\partial^n}{\partial x^n}
 \int_{x_L}\nolimits^x{\left(x-\xi\right)^{n-\alpha-1}}{u(\xi)}d\xi,
\end{equation*}
and
\begin{equation*}
 _{x}D_{x_R}^{\alpha}u(x)=
 \frac{1}{\Gamma(n-\alpha)}\frac{\partial^n}{\partial x^n}
\int_{x}\nolimits^{x_R}{\left(\xi-x\right)^{n-\alpha-1}}{u(\xi)}d\xi.
\end{equation*}
\end{definition}

\begin{lemma} [\cite{Ervin:06}] \label{lemma2.1}
 Let $\alpha>0$, $u \in C_0^{\infty}(\Omega)$, $\Omega \subset \mathbb{R}$, then
 \begin{equation*}
   \begin{split}
 &\mathcal{F}(_{-\infty}D_x^{\alpha}u(x))=(-i\omega)^{\alpha}\widehat{u}(\omega)~~{\rm and}~~
  \mathcal{F}(_{x}D_{\infty}^{\alpha}u(x))=(i\omega)^{\alpha}\widehat{u}(\omega),
   \end{split}
 \end{equation*}
 where $\mathcal{F}$ denotes Fourier transform operator and $\widehat{u}(\omega)=\mathcal{F}(u)$, i.e.,
 \begin{equation*}
    \widehat{u}(\omega)=\int_{\mathbb{R}}e^{i\omega x }u(x)dx.
 \end{equation*}
\end{lemma}

\begin{lemma}  \label{lemma2.3}
 Let the function $f(z)=\left(1+bz+cz^2+dz^3+ez^4\right)^\alpha$, then
$$f(0)=1,~~f'(0)=\alpha b,~~f''(0)=\alpha(\alpha-1)b^2+2\alpha c,$$
$$f'''(0)=\alpha(\alpha-1)(\alpha-2)b^3+6\alpha(\alpha-1)bc+6\alpha d,$$
$$f''''(0)=\alpha(\alpha-1)(\alpha-2)(\alpha-3)b^4+12\alpha(\alpha-1)(\alpha-2)b^2c+12\alpha(\alpha-1)c^2+24\alpha(\alpha-1)bd+24\alpha e,$$
and
 \begin{equation*}
    f(z)=f(0+z)=f(0)+zf'(0)+\frac{z^2}{2!}f''(0)+\frac{z^3}{3!}f'''(0)+\frac{z^4}{4!}f''''(0)+\mathcal{O}(z^5).
 \end{equation*}
\end{lemma}
\begin{proof}
  It is easy to check that by the  Taylor series expansion.
\end{proof}

\subsection{Derivation of the discretizations}

In this subsection, based on Lubich's operator (\ref{1.1}), we do the expansions to get the formulas of the coefficients when $\nu=3,4,5$; by the technique of Fourier transform, prove that the operators have their respective desired convergent order; and by weighting and shifting Lubich's 3rd (and 4th) order operators, obtain new high order discretization schemes.

First, taking  $\nu=1$, for all $|\zeta| \leq 1$, Eq. (\ref{1.1}) becomes the following equation  [\cite{Podlubny:99}],
\begin{equation}\label{2.1}
\begin{split}
(1-\zeta)^{\alpha}= \sum_{m=0}^{\infty}(-1)^m\left ( \begin{matrix}\alpha \\ m\end{matrix} \right )\zeta^m= \sum_{m=0}^{\infty}{l}_m^{1,\alpha}\zeta^m,
\end{split}
\end{equation}
with the recursively formula
\begin{equation}\label{2.2}
  {l}_0^{1,\alpha}=1, ~~~~{l}_m^{1,\alpha}=\left(1-\frac{\alpha+1}{m}\right)l_{m-1}^{1,\alpha},~~m \geq 1.
\end{equation}

Letting  $\nu=2$, for all $|\zeta| \leq 1$, Eq. (\ref{1.1}) has the following form  [\cite{Chen:13}],
\begin{equation}\label{2.3}
\begin{split}
\left(\frac{3}{2}-2\zeta+\frac{1}{2}\zeta^2\right)^{\alpha} =\left(\frac{3}{2}\right)^{\alpha}(1-\zeta)^{\alpha}(1-\frac{1}{3}\zeta)^{\alpha} =
\sum_{j=0}^{\infty}{l}_j^{2,\alpha}\zeta^{j},
\end{split}
\end{equation}
with
\begin{equation}\label{2.4}
 {l}_j^{2,\alpha}
  =\left(\frac{3}{2}\right)^{\alpha} \sum_{m=0}^{j} \mu_2^{m}\,{l}_m^{1,\alpha}\,{l}_{j-m}^{1,\alpha}, 
\end{equation}
where $\mu_2=\frac{1}{3}$ and ${l}_m^{1,\alpha}$ is defined by (\ref{2.2}).

Setting $\nu=3$, for all $|\zeta| \leq 1$, Eq. (\ref{1.1}) leads to the following form 
\begin{equation}\label{2.5}
\begin{split}
&\left(\frac{11}{6}-3\zeta+\frac{3}{2}\zeta^2-\frac{1}{3}\zeta^3\right)^{\alpha}\\
&\quad =\left(1-\zeta\right)^{\alpha}\left(\frac{11}{6}-\frac{7}{6}\zeta+\frac{1}{3}\zeta^2 \right)^{\alpha}\\
&\quad =\left(\frac{11}{6}\right)^{\alpha}\!\!\!\left(1-\zeta\right)^{\alpha}\left(1-{\mu_3}\zeta\right)^{\alpha}\left(1-\overline{{\mu_3}}\zeta\right)^{\alpha}\\
&\quad =\left(\frac{11}{6}\right)^{\alpha}\sum_{n=0}^{\infty}(-1)^n \left( \begin{matrix} \alpha \\n \end{matrix} \right )\zeta^n
        \sum_{m=0}^{\infty} \left(-{\mu_3}\right)^m  \left( \begin{matrix} \alpha \\m \end{matrix} \right )\zeta^m
        \sum_{l=0}^{\infty} \left(-\overline{{\mu_3}}\right)^l  \left( \begin{matrix} \alpha \\l \end{matrix} \right )\zeta^l \\
&\quad =\left(\frac{11}{6}\right)^{\alpha}\sum_{n=0}^{\infty}(-1)^n \left( \begin{matrix} \alpha \\n \end{matrix} \right )\zeta^n
\left\{ \sum_{j=0}^{\infty}\sum_{m=0}^{j} (-1)^j{\mu_3}\!^m\, \overline{{\mu_3}}^{j-m}
 \left( \begin{matrix} \alpha \\m \end{matrix} \right )\left( \begin{matrix} \alpha \\j-m \end{matrix} \right ) \zeta^{j}\right \}  \\
&\quad = \sum_{k=0}^{\infty}{l}_k^{3,\alpha}\zeta^{k},
\end{split}
\end{equation}
with
\begin{equation}\label{2.6}
 {l}_k^{3,\alpha}
  =\left(\frac{11}{6}\right)^{\alpha} \sum_{j=0}^{k}\sum_{m=0}^{j}{\mu_3}\!^m \,\overline{{\mu_3}}^{j-m}\,{l}_m^{1,\alpha}\,{l}_{j-m}^{1,\alpha}\,{l}_{k-j}^{1,\alpha},
\end{equation}
where ${\mu_3}=\frac{4}{7+\sqrt{39}\,i}$, $\overline{{\mu_3}}=\frac{4}{7-\sqrt{39}\,i}$, $i=\sqrt{-1}$ and $k=j+n$, $j=l+m$, and ${l}_m^{1,\alpha}$ is defined by (\ref{2.2}).

Taking $\nu=4$, for all $|\zeta| \leq 1$, Eq. (\ref{1.1}) reduces to the following form
\begin{equation*}
\begin{split}
\left(\frac{25}{12}-4\zeta+3\zeta^2-\frac{4}{3}\zeta^3+\frac{1}{4}\zeta^4\right)^{\alpha}
&=\left(1-\zeta\right)^{\alpha}\left(\frac{25}{12}-\frac{23}{12}\zeta+\frac{13}{12}\zeta^2-\frac{1}{4}\zeta^3 \right)^{\alpha}\\
&=\left(\frac{25}{12}\right)^{\alpha}\left(1-\zeta\right)^{\alpha}\left(-\frac{3}{25}\zeta^3 +\frac{13}{25}\zeta^2 -\frac{23}{25}\zeta + 1 \right)^{\alpha}\\
&=\left(\frac{25}{12}\right)^{\alpha}\left(1-\zeta\right)^{\alpha}\left(1-\nu_4\zeta\right)^{\alpha}
\left(1-{\mu_4}\zeta\right)^{\alpha}\left(1-\overline{{\mu_4}}\zeta\right)^{\alpha};
\end{split}
\end{equation*}
then from Shengjin's Formulas [\cite{Fan:89}],  we obtain
\begin{equation*}
\begin{split}
\nu_4&=\frac{3a}{-b-\left(\sqrt[3]{Y_1}-\sqrt[3]{-Y_2}\right)};\\
\mu_4&=\frac{3a}{-b+\frac{1}{2}\left(\sqrt[3]{Y_1}-\sqrt[3]{-Y_2}\right)+\frac{\sqrt{3}}{2}\left(\sqrt[3]{Y_1}+\sqrt[3]{-Y_2}\right)i};\\
\overline{{\mu_4}}&=\frac{3a}{-b+\frac{1}{2}\left(\sqrt[3]{Y_1}-\sqrt[3]{-Y_2}\right)-\frac{\sqrt{3}}{2}\left(\sqrt[3]{Y_1}+\sqrt[3]{-Y_2}\right)i},\\
\end{split}
\end{equation*}
where   $a=-\frac{3}{25}$,  $b=\frac{13}{25}$,  $c=-\frac{23}{25}$, $d=1$; $A=b^2-3ac$, $B=bc-9ad$, $C=c^2-3bd$; $\Delta=B^2-4AC>0$,
$Y_1=Ab+\frac{3}{2}a(-B-\sqrt{\Delta})>0$, $Y_2=Ab+\frac{3}{2}a(-B+\sqrt{\Delta})<0$, and $i=\sqrt{-1}$. Thus
\begin{equation}\label{2.7}
\begin{split}
&\left(\frac{25}{12}-4\zeta+3\zeta^2-\frac{4}{3}\zeta^3+\frac{1}{4}\zeta^4\right)^{\alpha}\\
&\quad =\left(\frac{25}{12}\right)^{\alpha}\sum_{n=0}^{\infty}(-1)^n \left( \begin{matrix} \alpha \\n \end{matrix} \right )\zeta^n
        \sum_{m=0}^{\infty} \left(-{\mu_4}\right)^m  \left( \begin{matrix} \alpha \\m \end{matrix} \right )\zeta^m
        \sum_{l=0}^{\infty} \left(-\overline{{\mu_4}}\right)^l  \left( \begin{matrix} \alpha \\l \end{matrix} \right )\zeta^l
        \sum_{p=0}^{\infty} \left(-{\nu_4}\right)^p  \left( \begin{matrix} \alpha \\p \end{matrix} \right )\zeta^p \\
&\quad = \sum_{n=0}^{\infty}{l}_n^{4,\alpha}\zeta^{n},
\end{split}
\end{equation}
with
\begin{equation}\label{2.8}
 {l}_n^{4,\alpha}
  =\left(\frac{25}{12}\right)^{\alpha}\sum_{k=0}^{n} \sum_{j=0}^{k}\sum_{m=0}^{j}{\nu_4}\!^{n-k} {\mu_4}\!^m \,\overline{{\mu_4}}^{j-m}\,{l}_m^{1,\alpha}\,{l}_{j-m}^{1,\alpha}\,{l}_{k-j}^{1,\alpha}\,{l}_{n-k}^{1,\alpha},
\end{equation}
and  ${l}_m^{1,\alpha}$ is defined by (\ref{2.2}).

Taking $\nu=5$, for all $|\zeta| \leq 1$, Eq. (\ref{1.1}) has the following form
\begin{equation*}
\begin{split}
&\left(\frac{137}{60}-5\zeta+5\zeta^2-\frac{10}{3}\zeta^3+\frac{5}{4}\zeta^4-\frac{1}{5}\zeta^5\right)^{\alpha}\\
&\quad =\left(1-\zeta\right)^{\alpha}\left(\frac{137}{60}-\frac{163}{60}\zeta+\frac{137}{60}\zeta^2-\frac{21}{20}\zeta^3 +\frac{1}{5}\zeta^4\right)^{\alpha}\\
&\quad =\left(\frac{137}{60}\right)^{\alpha}\left(1-\zeta\right)^{\alpha}\left(1-\frac{163}{137}\zeta+\zeta^2-\frac{63}{137}\zeta^3+\frac{12}{137}\zeta^4\right)^{\alpha}\\
&\quad =\left(\frac{1}{5}\right)^{\alpha}\left(1-\zeta\right)^{\alpha}\left(\zeta^4 -\frac{21}{4}\zeta^3 +\frac{137}{12}\zeta^2 -\frac{163}{12}\zeta +\frac{137}{12} \right)^{\alpha}\\
&\quad=\left(\frac{137}{60}\right)^{\alpha}\left(1-\zeta\right)^{\alpha}\left(1-\nu_5\zeta\right)^{\alpha}\left(1-\overline{\nu_5}\zeta\right)^{\alpha}
\left(1-{\mu_5}\zeta\right)^{\alpha}\left(1-\overline{{\mu_5}}\zeta\right)^{\alpha}.
\end{split}
\end{equation*}
According to Ferrari's Formulas [\cite{Polyanin:07}] and Shengjin's Formulas [\cite{Fan:89}],  we obtain
\begin{equation*}
\begin{split}
\nu_5           &=\frac{4}{-(b+M) + \sqrt{(b+M)^2-16(y+\frac{N}{M})}};\\
\overline{\nu_5}&=\frac{4}{-(b-M) + \sqrt{(b-M)^2-16(y-\frac{N}{M})}};\\
\mu_5           &=\frac{4}{-(b+M) - \sqrt{(b+M)^2-16(y+\frac{N}{M})}};\\
\overline{{\mu_5}}&=\frac{4}{-(b-M) - \sqrt{(b-M)^2-16(y-\frac{N}{M})}},\\
\end{split}
\end{equation*}
where $a=1$,  $b=-\frac{21}{4}$,  $c=\frac{137}{12}$, $d=-\frac{163}{12}$, $e=\frac{137}{12}$;
$\widetilde{a}=1$,  $\widetilde{b}=-\frac{137}{24}$,  $\widetilde{c}=\frac{1231}{192}$, $\widetilde{d}=\frac{4259}{1536}$;
$A=\widetilde{b}^2-3\widetilde{a}\widetilde{c}$, $B=\widetilde{b}\widetilde{c}-9\widetilde{a}\widetilde{d}$,
$C=\widetilde{c}^2-3\widetilde{b}\widetilde{d}$; $\Delta=B^2-4AC<0$;
 $T=\frac{2A\widetilde{b}-3\widetilde{a}B}{2A^{\frac{3}{2}}} $, $\theta=\arccos(T)$  $y=\frac{-\widetilde{b}-2A^{\frac{1}{2}}\cos\frac{\theta}{3} }{3a}$;
and $M=\sqrt{8y+b^2-4c}$,  $N=\sqrt{by-d}$. Therefore, we have
\begin{equation}\label{2.9}
\begin{split}
&\left(\frac{137}{60}-5\zeta+5\zeta^2-\frac{10}{3}\zeta^3+\frac{5}{4}\zeta^4-\frac{1}{5}\zeta^5\right)^{\alpha}
 = \sum_{q=0}^{\infty}{l}_q^{5,\alpha}\zeta^{q},
\end{split}
\end{equation}
where
\begin{equation}\label{2.10}
 {l}_q^{5,\alpha}
  =\left(\frac{137}{60}\right)^{\alpha}\sum_{n=0}^{q}\sum_{k=0}^{n} \sum_{j=0}^{k}\sum_{m=0}^{j}\overline{\nu_5}\!^{q-n}{\nu_5}\!^{n-k} {\mu_5}\!^m \,\overline{{\mu_5}}^{j-m}{\,l}_m^{1,\alpha}{\,l}_{j-m}^{1,\alpha}{\,l}_{k-j}^{1,\alpha}{\,l}_{n-k}^{1,\alpha}{\,l}_{q-n}^{1,\alpha},
\end{equation}
and  ${l}_m^{1,\alpha}$ is defined by (\ref{2.2}).

In the following, by using the technique of Fourier transform, we again list and simply prove the convergent order of Lubich's operator. 
\begin{lemma}(Case $\nu=1$; \cite{Meerschaert:04})\label{lemma2.4}
   Let $u $, $_{-\infty}D_x^{\alpha+1}u(x)$  with $\alpha \in (1,2)$  and their Fourier transforms belong to $L_1(\mathbb{R})$, and  denote that
 \begin{equation*}
_{1L}A_{p}^{1,\alpha}u(x)=\frac{1}{h^{\alpha}}\sum_{m=0}^{\infty}{l}_m^{1,\alpha}u(x-(m-p)h),
\end{equation*}
 where ${l}_m^{1,\alpha}$ is defined by (\ref{2.2}) and $p$ an integer.  Then
$$
 _{-\infty}D_x^{\alpha}u(x)=\,_{1L}A_{p}^{1,\alpha}u(x)+\mathcal{O}(h).
$$
\end{lemma}
\begin{lemma}(Case $\nu=2$; \cite{Chen:13})\label{lemma2.5}
  Let $u $, $_{-\infty}D_x^{\alpha+1}u(x)$ (or $_{-\infty}D_x^{\alpha+2}u(x)$) with $\alpha \in (1,2)$  and their Fourier transforms belong to $L_1(\mathbb{R})$ when $p\neq 0$ (or $p=0$); and  denote that
 \begin{equation*}
_{1L}A_{p}^{2,\alpha}u(x)=\frac{1}{h^{\alpha}}\sum_{j=0}^{\infty}{l}_j^{2,\alpha}u(x-(j-p)h),
\end{equation*}
where ${l}_j^{2,\alpha}$ is defined by (\ref{2.4}) and $p$ an integer. Then
\begin{equation*}
\begin{split}
 &_{-\infty}D_x^{\alpha}u(x)=\,_{1L}A_{p}^{2,\alpha}u(x)+\mathcal{O}(h),   \,~~~~p\neq 0;\\
  &_{-\infty}D_x^{\alpha}u(x)=\,_{1L}A_{p}^{2,\alpha}u(x)+\mathcal{O}(h^2),   ~~~p=0.
 \end{split}
\end{equation*}
\end{lemma}

\begin{lemma}(Case $\nu=3$)\label{lemma2.6}
  Let $u $, $_{-\infty}D_x^{\alpha+1}u(x)$ (or $_{-\infty}D_x^{\alpha+3}u(x)$) and their Fourier transforms belong to $L_1(\mathbb{R})$ when $p\neq 0$ (or $p=0$); and  denote that
 \begin{equation}\label{2.11}
_{1L}A_{p}^{3,\alpha}u(x)=\frac{1}{h^{\alpha}}\sum_{k=0}^{\infty}{l}_k^{3,\alpha}u(x-(k-p)h),
\end{equation}
where ${l}_k^{3,\alpha}$ is defined by (\ref{2.6}) and $p$ an integer. Then
\begin{equation*}
\begin{split}
 &_{-\infty}D_x^{\alpha}u(x)=\,_{1L}A_{p}^{3,\alpha}u(x)+\mathcal{O}(h),   \,~~~~p\neq 0;\\
  &_{-\infty}D_x^{\alpha}u(x)=\,_{1L}A_{p}^{3,\alpha}u(x)+\mathcal{O}(h^3),   ~~~p=0.
 \end{split}
\end{equation*}
\end{lemma}

\begin{proof}
According to (\ref{2.5}), there exist 
\begin{equation*}
\begin{split}
\mathcal{F}(_{1L}A_{p}^{3,\alpha}u)(\omega)&=h^{-\alpha}\sum_{k=0}^{\infty}{l}_k^{3,\alpha} \mathcal{F}\left(u(x-(k-p)h)\right)(\omega) \\
&=h^{-\alpha}e^{-i\omega ph} \sum_{k=0}^{\infty}{l}_k^{3,\alpha} \left(e^{i\omega h}\right)^k \widehat{u}(\omega)\\
&=h^{-\alpha}e^{-i\omega ph}\left(1-e^{i\omega h}\right)^{\alpha}
  \left(\frac{11}{6}-\frac{7}{6}e^{i\omega h}+\frac{1}{3}e^{2i\omega h} \right)^{\alpha}\widehat{u}(\omega)\\
&=(-i\omega)^{\alpha}e^{-i\omega ph} \left(\frac{1-e^{i\omega h}}{-i\omega h}\right )^{\alpha}
\left(\frac{11}{6}-\frac{7}{6}e^{i\omega h}+\frac{1}{3}e^{2i\omega h} \right)^{\alpha}\widehat{u}(\omega)\\
&=(-i\omega)^{\alpha}e^{ pz} \left(\frac{1-e^{-z}}{z}\right )^{\alpha}
\left(\frac{11}{6}-\frac{7}{6}e^{-z}+\frac{1}{3}e^{-2z} \right)^{\alpha}\widehat{u}(\omega),
  \end{split}
\end{equation*}
with $z=-i\omega h$. It is easy to check that
 \begin{equation}\label{2.12}
\begin{split}
&e^{ pz} \left(\frac{1-e^{-z}}{z}\right )^{\alpha}\\
&\quad =1+\Big(p-\frac{\alpha}{2}\Big)z+ \Big(\frac{1}{2}p^2 -\frac{\alpha }{2}p+\frac{3\alpha^2+\alpha}{24}\Big) z^2 \\
&\qquad +\Big(\frac{1}{6}p^3 -\frac{\alpha }{4}p^2
+\frac{3\alpha^2+\alpha}{24}p-\frac{\alpha^3+\alpha^2}{48} \Big)z^3   \\
&\qquad +\Big(\frac{1}{24}p^4-\frac{\alpha}{12}p^3 +\frac{3\alpha^2+\alpha }{48}p^2
-\frac{\alpha^3+\alpha^2}{48}p+\frac{15\alpha^4+30\alpha^3+5\alpha^2-2\alpha}{5760} \Big)z^4  +\mathcal{O}(z^5),
  \end{split}
\end{equation}
and from the Lemma \ref{lemma2.3}, we have
 \begin{equation*}
\begin{split}
\left(\frac{11}{6}-\frac{7}{6}e^{-z}+\frac{1}{3}e^{-2z} \right)^{\alpha}
& =\left(1+\frac{1}{2}z+\frac{1}{12}z^2-\frac{1}{4}z^3+\mathcal{O}(z^4)\right)^{\alpha}\\
&=1+ \frac{\alpha}{2}z+\frac{\alpha(3\alpha-1)}{24}z^2+\frac{\alpha(\alpha+3)(\alpha-4)}{48}z^3+\mathcal{O}(z^4),
  \end{split}
\end{equation*}
then from the above equation and  (\ref{2.12})   we obtain
\begin{equation}\label{2.13}
\begin{split}
&e^{ pz} \left(\frac{1-e^{-z}}{z}\right )^{\alpha}\left(\frac{11}{6}-\frac{7}{6}e^{-z}+\frac{1}{3}e^{-2z} \right)^{\alpha}
\!\!  =1+pz+\frac{p^2}{2}z^2+\frac{2p^3-3\alpha}{12}z^3+\mathcal{O}(z^4).
  \end{split}
\end{equation}
Therefore, from Lemma \ref{lemma2.1}, we obtain
\begin{equation*}
\begin{split}
\mathcal{F}(_{1L}A_{p}^{3,\alpha}u)(\omega)=\mathcal{F}(_{-\infty}D_x^{\alpha}u(x))+ \widehat{\phi}(\omega),
  \end{split}
\end{equation*}
where $ \widehat{\phi}(\omega)=(-i\omega)^{\alpha}\left(pz+\frac{p^2}{2}z^2+\frac{2p^3-3\alpha}{12}z^3+\mathcal{O}(z^4)\right)\widehat{u}(\omega)$,
$z=-i\omega h$. Then there exists
\begin{equation*}
\begin{split}
&|\widehat{\phi}(\omega)| \leq \widetilde{c}|i\omega|^{\alpha+1}|\widehat{u}(\omega)|\cdot h, ~~~~~~p\neq 0;\\
&|\widehat{\phi}(\omega)| \leq c|i\omega|^{\alpha+3}|\widehat{u}(\omega)|\cdot h^3, ~\,~~~p=0.
\end{split}
\end{equation*}
Hence
\begin{equation*}
\begin{split}
|_{-\infty}D_x^{\alpha}u(x)-\,_{1L}A_{p}^{3,\alpha}u(x)|=|\phi(x)| \leq \frac{1}{2\pi}\int_{\mathbb{R}}|\widehat{\phi}(\omega)|dx= \left\{ \begin{array}
 {l@{\quad} l}
\mathcal{O}(h),&p\neq 0;\\
\mathcal{O}(h^3),& p=0.
 \end{array}
 \right.
  \end{split}
\end{equation*}
\end{proof}

\begin{lemma}(Case $\nu=4$)\label{lemma2.7}
 Let $u $, $_{-\infty}D_x^{\alpha+1}u(x)$ (or $_{-\infty}D_x^{\alpha+4}u(x)$) with $\alpha \in (1,2)$  and their Fourier transforms belong to $L_1(\mathbb{R})$ when $p\neq 0$ (or $p=0$); and  denote that
 \begin{equation}\label{2.14}
_{1L}A_{p}^{4,\alpha}u(x)=\frac{1}{h^{\alpha}}\sum_{n=0}^{\infty}{l}_n^{4,\alpha}u(x-(n-p)h),
\end{equation}
where ${l}_n^{4,\alpha}$ is defined by (\ref{2.8}) and $p$ an integer. Then
\begin{equation*}
\begin{split}
 &_{-\infty}D_x^{\alpha}u(x)=\,_{1L}A_{p}^{4,\alpha}u(x)+\mathcal{O}(h),   \,~~~~p\neq 0;\\
  &_{-\infty}D_x^{\alpha}u(x)=\,_{1L}A_{p}^{4,\alpha}u(x)+\mathcal{O}(h^4),   ~~~p=0.
 \end{split}
\end{equation*}
\end{lemma}

\begin{proof}
According to (\ref{2.7})   we obtain
\begin{equation*}
\begin{split}
\mathcal{F}(_{1L}A_{p}^{4,\alpha}u)(\omega)&=h^{-\alpha}\sum_{n=0}^{\infty}{l}_n^{4,\alpha} \mathcal{F}\left(u(x-(n-p)h)\right)(\omega) \\
&=h^{-\alpha}e^{-i\omega ph} \sum_{n=0}^{\infty}{l}_n^{4,\alpha} \left(e^{i\omega h}\right)^n \widehat{u}(\omega)\\
&=h^{-\alpha}e^{-i\omega ph}\left(1-e^{i\omega h}\right)^{\alpha}
\left(\frac{25}{12}-\frac{23}{12}e^{i\omega h}+\frac{13}{12}e^{2i\omega h}-\frac{1}{4}e^{3i\omega h} \right)^{\alpha}\widehat{u}(\omega)\\
&=(-i\omega)^{\alpha}e^{-i\omega ph} \left(\frac{1-e^{i\omega h}}{-i\omega h}\right )^{\alpha}
\left(\frac{25}{12}-\frac{23}{12}e^{i\omega h}+\frac{13}{12}e^{2i\omega h}-\frac{1}{4}e^{3i\omega h} \right)^{\alpha}\widehat{u}(\omega)\\
&=(-i\omega)^{\alpha}e^{ pz} \left(\frac{1-e^{-z}}{z}\right )^{\alpha}
\left(\frac{25}{12}-\frac{23}{12}e^{-z}+\frac{13}{12}e^{-2z}-\frac{1}{4}e^{-3z} \right)^{\alpha}\widehat{u}(\omega),
  \end{split}
\end{equation*}
where $z=-i\omega h$. From Lemma \ref{lemma2.3}, there exist 
 \begin{equation*}
\begin{split}
\left(\frac{25}{12}-\frac{23}{12}e^{-z}+\frac{13}{12}e^{-2z}-\frac{1}{4}e^{-3z} \right)^{\alpha}
& =\left(1+\frac{1}{2}z+\frac{1}{12}z^2-\frac{29}{144}z^4+\mathcal{O}(z^5)\right)^{\alpha}\\
&=1+ \frac{\alpha}{2}z+\frac{\alpha(3\alpha-1)}{24}z^2+\frac{\alpha^2(\alpha-1)}{48}z^3+\mathcal{O}(z^4),
  \end{split}
\end{equation*}
then we obtain
\begin{equation}\label{2.15}
\begin{split}
&e^{ pz} \left(\frac{1-e^{-z}}{z}\right )^{\alpha}\left(\frac{25}{12}-\frac{23}{12}e^{-z}+\frac{13}{12}e^{-2z}-\frac{1}{4}e^{-3z} \right)^{\alpha}
\!\!  =1+pz+\frac{p^2}{2!}z^2+\frac{p^3}{3!}z^3+\mathcal{O}(z^4).
  \end{split}
\end{equation}
Therefore, from Lemma \ref{lemma2.1}, we obtain
\begin{equation*}
\begin{split}
\mathcal{F}(_{1L}A_{p}^{4,\alpha}u)(\omega)=\mathcal{F}(_{-\infty}D_x^{\alpha}u(x))+ \widehat{\phi}(\omega),
  \end{split}
\end{equation*}
where $ \widehat{\phi}(\omega)=(-i\omega)^{\alpha}\left(pz+\frac{p^2}{2!}z^2+\frac{p^3}{3!}z^3+\mathcal{O}(z^4)\right)\widehat{u}(\omega)$.
Then there exists
\begin{equation*}
\begin{split}
&|\widehat{\phi}(\omega)| \leq \widetilde{c}|i\omega|^{\alpha+1}|\widehat{u}(\omega)|\cdot h, ~~~~~~p\neq 0;\\
&|\widehat{\phi}(\omega)| \leq c|i\omega|^{\alpha+4}|\widehat{u}(\omega)|\cdot h^4, ~\,~~~p=0.
\end{split}
\end{equation*}
Hence
\begin{equation*}
\begin{split}
|_{-\infty}D_x^{\alpha}u(x)-\,_{1L}A_{p}^{4,\alpha}u(x)|=|\phi(x)| \leq \frac{1}{2\pi}\int_{\mathbb{R}}|\widehat{\phi}(\omega)|dx= \left\{ \begin{array}
 {l@{\quad} l}
\mathcal{O}(h),&p\neq 0;\\
\mathcal{O}(h^4),& p=0.
 \end{array}
 \right.
  \end{split}
\end{equation*}
\end{proof}

\begin{lemma}(Case $\nu=5$)\label{lemma2.8}
Let $u$, $_{-\infty}D_x^{\alpha+1}u(x)$ (or $_{-\infty}D_x^{\alpha+5}u(x)$) with $\alpha \in (1,2)$  and their Fourier transforms belong to $L_1(\mathbb{R})$ when $p\neq 0$ (or $p=0$); and  denote that
 \begin{equation}\label{2.16}
_{1L}A_{p}^{5,\alpha}u(x)=\frac{1}{h^{\alpha}}\sum_{q=0}^{\infty}{l}_q^{5,\alpha}u(x-(q-p)h),
\end{equation}
where ${l}_q^{5,\alpha}$ is defined by (\ref{2.10}) and $p$ an integer. Then
\begin{equation*}
\begin{split}
 &_{-\infty}D_x^{\alpha}u(x)=\,_{1L}A_{p}^{5,\alpha}u(x)+\mathcal{O}(h),   \,~~~~p\neq 0;\\
  &_{-\infty}D_x^{\alpha}u(x)=\,_{1L}A_{p}^{5,\alpha}u(x)+\mathcal{O}(h^5),   ~~~p=0.
 \end{split}
\end{equation*}
\end{lemma}

\begin{proof}
According to (\ref{2.9})   we obtain
\begin{equation*}
\begin{split}
\mathcal{F}(_{1L}A_{p}^{5,\alpha}u)(\omega)&=h^{-\alpha}\sum_{l=0}^{\infty}{l}_q^{5,\alpha} \mathcal{F}\left(u(x-(q-p)h)\right)(\omega) \\
&=h^{-\alpha}e^{-i\omega ph} \sum_{l=0}^{\infty}{l}_q^{5,\alpha} \left(e^{i\omega h}\right)^q \widehat{u}(\omega)\\
&=h^{-\alpha}e^{-i\omega ph}\left(1-e^{i\omega h}\right)^{\alpha}
\left(\frac{137}{60}-\frac{163}{60}e^{i\omega h}+\frac{137}{60}e^{2i\omega h}
  -\frac{21}{20}e^{3i\omega h} +\frac{1}{5}e^{4i\omega h}\right)^{\alpha}\widehat{u}(\omega)\\
&=(-i\omega)^{\alpha}e^{ pz} \left(\frac{1-e^{-z}}{z}\right )^{\alpha}
\left(\frac{137}{60}-\frac{163}{60}e^{-z}+\frac{137}{60}e^{-2z}
  -\frac{21}{20}e^{-3z} +\frac{1}{5}e^{-4z}\right)^{\alpha}\widehat{u}(\omega),
  \end{split}
\end{equation*}
where $z=-i\omega h$. Using Lemma \ref{lemma2.3} leads to 
 \begin{equation*}
\begin{split}
&\left(\frac{137}{60}-\frac{163}{60}e^{-z}+\frac{137}{60}e^{-2z}-\frac{21}{20}e^{-3z} +\frac{1}{5}e^{-4z}\right)^{\alpha}\\
& =\left(1+\frac{1}{2}z+\frac{1}{12}z^2-\frac{1}{720}z^4+\mathcal{O}(z^5)\right)^{\alpha}\\
&=1+ \frac{\alpha}{2}z+\frac{\alpha(3\alpha-1)}{24}z^2+\frac{\alpha^2(\alpha-1)}{48}z^3
+\frac{15\alpha^4-30\alpha^3+5\alpha^2+2\alpha}{5760} z^4+\mathcal{O}(z^5),
  \end{split}
\end{equation*}
then we obtain
\begin{equation}\label{2.17}
\begin{split}
&e^{ pz} \left(\frac{1-e^{-z}}{z}\right )^{\alpha}\left(\frac{137}{60}-\frac{163}{60}e^{-z}+\frac{137}{60}e^{-2z}-\frac{21}{20}e^{-3z} +\frac{1}{5}e^{-4z}\right)^{\alpha}\\
&\quad   =1+pz+\frac{p^2}{2!}z^2+\frac{p^3}{3!}z^3+\frac{p^4}{4!}z^4+\mathcal{O}(z^5).
  \end{split}
\end{equation}
Therefore, from Lemma \ref{lemma2.1}, there exists 
\begin{equation*}
\begin{split}
\mathcal{F}(_{1L}A_{p}^{5,\alpha}u)(\omega)=\mathcal{F}(_{-\infty}D_x^{\alpha}u(x))+ \widehat{\phi}(\omega),
  \end{split}
\end{equation*}
where $ \widehat{\phi}(\omega)=(-i\omega)^{\alpha}\left(pz+\frac{p^2}{2!}z^2+\frac{p^3}{3!}z^3+\frac{p^4}{4!}z^4+\mathcal{O}(z^5)\right)\widehat{u}(\omega)$.
Then we get 
\begin{equation*}
\begin{split}
&|\widehat{\phi}(\omega)| \leq \widetilde{c}|i\omega|^{\alpha+1}|\widehat{u}(\omega)|\cdot h, ~~~~~~p\neq 0;\\
&|\widehat{\phi}(\omega)| \leq c|i\omega|^{\alpha+5}|\widehat{u}(\omega)|\cdot h^5, ~\,~~~p=0.
\end{split}
\end{equation*}
Hence
\begin{equation*}
\begin{split}
|_{-\infty}D_x^{\alpha}u(x)-\,_{1L}A_{p}^{5,\alpha}u(x)|=|\phi(x)| \leq \frac{1}{2\pi}\int_{\mathbb{R}}|\widehat{\phi}(\omega)|dx= \left\{ \begin{array}
 {l@{\quad} l}
\mathcal{O}(h),&p\neq 0;\\
\mathcal{O}(h^5),& p=0.
 \end{array}
 \right.
  \end{split}
\end{equation*}
\end{proof}

According to Lemmas \ref{lemma2.3}-\ref{lemma2.7}, the fractional approximation operators have the same form
$$_{1L}A_{p}^{\nu,\alpha}u(x)=\frac{1}{h^{\alpha}}\sum_{k=0}^{\infty}{l}_k^{\nu,\alpha}u(x-(k-p)h),~~\nu=1,2,3,4,5,$$
where $l_k^{1,\alpha}$, $l_k^{2,\alpha}$, $l_k^{3,\alpha}$, $l_k^{4,\alpha}$ and $l_k^{5,\alpha}$ are, respectively,
defined by (\ref{2.2}), (\ref{2.4}), (\ref{2.6}), (\ref{2.8})  and (\ref{2.10}). 
By the same idea of the proof in [\cite{Chen:13}],
we can get the following Theorems \ref{theorem2.2}-\ref{theorem2.7}; and for the simplicity, we omit the proofs here. 
\begin{theorem}\label{theorem2.2}(Case $\nu=3,4$; Second order approximations for left Riemann-Liouville  derivative)
 Let $u$, $_{-\infty}D_x^{\alpha+2}u(x)$ with $\alpha \in (1,2)$ and their Fourier transforms  belong to $L_1(\mathbb{R})$. Denote that
 \begin{equation}\label{2.18}
_{2L}A_{p,q}^{\nu,\alpha}u(x)=w_p^\nu\,_{1L}A_{p}^{\nu,\alpha}u(x)+w_q^\nu\,_{1L}A_{q}^{\nu,\alpha}u(x),~~\nu=3,4,
\end{equation}
where $_{1L}A_{p}^{3,\alpha}$ and $_{1L}A_{q}^{3,\alpha}$ are defined in (\ref{2.11}); $_{1L}A_{p}^{4,\alpha}$ and $_{1L}A_{q}^{4,\alpha}$ are defined in (\ref{2.14});  $w_p^\nu=\frac{q}{q-p}$,  $w_q^\nu=\frac{p}{p-q}$,  $\nu=3,4$,
 and   $p$, $q$ are integers, $p \neq  q$. Then
\begin{equation*}
\begin{split}
  &_{-\infty}D_x^{\alpha}u(x)=\,_{2L}A_{p,q}^{\nu,\alpha}u(x)+\mathcal{O}(h^2).
 \end{split}
\end{equation*}
\end{theorem}

\begin{theorem}\label{theorem2.6}(Case $\nu=3,4$;  Third order approximations for left Riemann-Liouville  derivative)
 Let $u$, $_{-\infty}D_x^{\alpha+3}u(x)$ with $\alpha \in (1,2)$ and their Fourier transforms  belong to $L_1(\mathbb{R})$. Denote that
 \begin{equation}\label{2.19}
_{3L}A_{p,q,r,s}^{\nu,\alpha}u(x)=w_{p,q}^\nu\,_{2L}A_{p,q}^{\nu,\alpha}u(x)+w_{r,s}^\nu\,_{2L}A_{r,s}^{\nu,\alpha}u(x),~~~~\nu=3,4,
\end{equation}
where $_{2L}A_{p,q}^{\nu,\alpha}$ and $_{2L}A_{r,s}^{\nu,\alpha}$  are defined in (\ref{2.18}),
 $w_{p,q}^\nu=\frac{rs}{rs-pq}$, $w_{r,s}^\nu=\frac{pq}{pq-rs}$, $\nu=3,4$,
 and  $p$, $q$, $r$, $s$ are  integers, $pq \neq rs$. Then
\begin{equation*}
\begin{split}
  &_{-\infty}D_x^{\alpha}u(x)=\,_{3L}A_{p,q,r,s}^{\nu,\alpha}u(x)+\mathcal{O}(h^3).
 \end{split}
\end{equation*}
\end{theorem}

\begin{theorem}\label{theorem2.7}(Case $\nu=3,4$; Fourth order approximations for left Riemann-Liouville  derivative)
 Let $u$, $_{-\infty}D_x^{\alpha+4}u(x)$ with $\alpha \in (1,2)$ and their Fourier transforms  belong to $L_1(\mathbb{R})$. Denote that
 \begin{equation}\label{2.20}
_{4L}A_{p,q,r,s,\overline{p},\overline{q},\overline{r},\overline{s}}^{\nu,\alpha}u(x)=w_{p,q,r,s}^\nu\,_{3L}A_{p,q,r,s}^{\nu,\alpha}u(x)
+w^\nu_{\overline{p},\overline{q},\overline{r},\overline{s}}\,_{3L}A_{\overline{p},\overline{q},\overline{r},\overline{s}}^{\nu,\alpha}u(x),~~\nu=3,4,
\end{equation}
where $_{3L}A_{p,q,r,s}^{\nu,\alpha}$ and $_{3L}A_{\overline{p},\overline{q},\overline{r},\overline{s}}^{\nu,\alpha}$  are defined in (\ref{2.19}), and
\begin{equation*}
\begin{split}
w^\nu_{p,q,r,s}=\frac{c^\nu_{\overline{p},\overline{q},\overline{r},\overline{s}}}
{c^\nu_{\overline{p},\overline{q},\overline{r},\overline{s}}-c^\nu_{p,q,r,s}}; \quad
w^\nu_{\overline{p},\overline{q},\overline{r},\overline{s}}
=\frac{c^\nu_{p,q,r,s}}
{c^\nu_{p,q,r,s}-c^\nu_{\overline{p},\overline{q},\overline{r},\overline{s}}};
 \end{split}
\end{equation*}
where
\begin{equation*}
\begin{split}
c^3_{p,q,r,s}=\frac{2pqrs(r+s-p-q)+3\alpha(pq-rs)}{12(rs-pq)};\quad
c^3_{\overline{p},\overline{q},\overline{r},\overline{s}}=
\frac{2\overline{p}\,\overline{q}\,\overline{r}\,\overline{s}(\overline{r}+\overline{s}-\overline{p}-\overline{q})
+3\alpha(\overline{p}\,\overline{q}-\overline{r}\,\overline{s})}{12(\overline{r}\,\overline{s}-\overline{p}\,\overline{q})};
 \end{split}
\end{equation*}
and
\begin{equation*}
\begin{split}
c^4_{p,q,r,s}=\frac{pqrs(r+s-p-q)}{6(rs-pq)};\quad
c^4_{\overline{p},\overline{q},\overline{r},\overline{s}}=
\frac{\overline{p}\,\overline{q}\,\overline{r}\,\overline{s}(\overline{r}+\overline{s}-\overline{p}-\overline{q})}
{6(\overline{r}\,\overline{s}-\overline{p}\,\overline{q})};
 \end{split}
\end{equation*}
 and  $p$, $q$, $r$, $s$;  $\overline{p}$, $\overline{q}$, $\overline{r}$, $\overline{s}$, are  integers,
 $ c^\nu_{p,q,r,s} \neq c^\nu_{\overline{p},\overline{q},\overline{r},\overline{s}}$. Then
\begin{equation*}
\begin{split}
  _{-\infty}D_x^{\alpha}u(x)=\,_{4L}A_{p,q,r,s,\overline{p},\overline{q},\overline{r},\overline{s}}^{\nu,\alpha}u(x)+\mathcal{O}(h^4).
 \end{split}
\end{equation*}
\end{theorem}


For the right Riemann-Liouville fractional derivative, denote that
 \begin{equation}\label{2.21}
_{1R}A_{p}^{\nu,\alpha}u(x)=\frac{1}{h^{\alpha}}\sum_{k=0}^{\infty}{l}_k^{\nu,\alpha}u(x+(k-p)h),~~\nu=1,2,3,4,5,
\end{equation}
where $l_k^{1,\alpha}$, $l_k^{2,\alpha}$, $l_k^{3,\alpha}$, $l_k^{4,\alpha}$ and $l_k^{5,\alpha}$ are, respectively,
defined by (\ref{2.2}), (\ref{2.4}), (\ref{2.6}), (\ref{2.8})  and (\ref{2.10}), and $p$ is a integer.  In particular,
the coefficients in (\ref{2.22}) are completely  the same as the ones in (\ref{2.18}); and the coefficients in (\ref{2.23}) the same as
the ones in (\ref{2.19}); and the coefficients in (\ref{2.24}) the same as
the ones in (\ref{2.20}).

\begin{theorem}\label{theorem2.5}(Case $\nu=3,4$; Second order approximations for right Riemann-Liouville  derivative)
 Let $u$, $_xD_{\infty}^{\alpha+2}u(x)$ with $\alpha \in (1,2)$ and their Fourier transforms  belong to $L_1(\mathbb{R})$, and  denote that
 \begin{equation}\label{2.22}
_{2R}A_{p,q}^{\nu,\alpha}u(x)=w_p^\nu\,_{1R}A_{p}^{\nu,\alpha}u(x)+w_q^\nu\,_{1R}A_{q}^{\nu,\alpha}u(x),~~\nu=3,4,
\end{equation}
 then
\begin{equation*}
\begin{split}
  &_xD_{\infty}^{\alpha}u(x)=\,_{2R}A_{p,q}^{\nu,\alpha}u(x)+\mathcal{O}(h^2).
 \end{split}
\end{equation*}
\end{theorem}

\begin{theorem}\label{theorem2.6}(Case $\nu=3,4$;  Third order approximations for right Riemann-Liouville  derivative)
 Let $u$, $_xD_{\infty}^{\alpha+3}u(x)$ with $\alpha \in (1,2)$ and their Fourier transforms  belong to $L_1(\mathbb{R})$, and  denote that
 \begin{equation}\label{2.23}
_{3R}A_{p,q,r,s}^{\nu,\alpha}u(x)=w_{p,q}^\nu\,_{2R}A_{p,q}^{\nu,\alpha}u(x)+w_{r,s}^\nu\,_{2R}A_{r,s}^{\nu,\alpha}u(x),~~\nu=3,4,
\end{equation}
then
\begin{equation*}
\begin{split}
  &_xD_{\infty}^{\alpha}u(x)=\,_{3R}A_{p,q,r,s}^{\nu,\alpha}u(x)+\mathcal{O}(h^3).
 \end{split}
\end{equation*}
\end{theorem}

\begin{theorem}\label{theorem2.7}(Case $\nu=3,4$; Fourth order approximations for right Riemann-Liouville  derivative)
 Let $u$, $_xD_{\infty}^{\alpha+4}u(x)$ with $\alpha \in (1,2)$ and their Fourier transforms  belong to $L_1(\mathbb{R})$, and  denote that
 \begin{equation}\label{2.24}
_{4R}A_{p,q,r,s,\overline{p},\overline{q},\overline{r},\overline{s}}^{\nu,\alpha}u(x)=w_{p,q,r,s}^\nu\,_{3R}A_{p,q,r,s}^{\nu,\alpha}u(x)
+w^\nu_{\overline{p},\overline{q},\overline{r},\overline{s}}\,_{3R}A_{\overline{p},\overline{q},\overline{r},\overline{s}}^{\nu,\alpha}u(x),~~\nu=3,4,
\end{equation}
then
\begin{equation*}
\begin{split}
  _xD_{\infty}^{\alpha}u(x)=\,_{4R}A_{p,q,r,s,\overline{p},\overline{q},\overline{r},\overline{s}}^{\nu,\alpha}u(x)+\mathcal{O}(h^4).
 \end{split}
\end{equation*}
\end{theorem}


All the above schemes are applicable to finite domain, say, $(x_L, x_R)$, after performing zero extensions to the functions considered. Let $u(x)$ be the zero extended function from the finite domain $(x_L, x_R)$, and satisfy the requirements of the above corresponding theorems (Lemma \ref{lemma2.6} - Theorem \ref{theorem2.7}). Taking
 \begin{equation}\label{2.25}
  \begin{split}
&_{1L}\widetilde{A}_{p}^{\nu,\alpha}u(x)=\frac{1}{h^{\alpha}}\sum_{k=0}^{[\frac{x-x_L}{h}]+p}l_k^{\nu,\alpha}u(x-(k-p)h), ~~\nu=3,4.
\end{split}
\end{equation}
 Then
\begin{equation}\label{2.26}
  \begin{split}
    _{x_L}D_{x}^{\alpha}u(x)&=\,_{1L}\widetilde{A}_{p}^{\nu,\alpha}u(x)+\mathcal{O}(h),~~~~~~~~~~~~~~~~p \neq 0;\\
_{x_L}D_{x}^{\alpha}u(x)&=\,_{1L}\widetilde{A}_{p}^{\nu,\alpha}u(x)+\mathcal{O}(h^\nu),~~~~~~~~~~~~~~p = 0;\\
_{x_L}D_{x}^{\alpha}u(x)&=\,_{2L}\widetilde{A}_{p,q}^{\nu,\alpha}u(x)+\mathcal{O}(h^2),~~~~~~~~~~~~~~
 _{2L}\widetilde{A}_{p,q}^{\nu,\alpha}u(x)=w_p^\nu\,_{1L}\widetilde{A}_{p}^{\nu,\alpha}u(x)+w_q^\nu\,_{1L}\widetilde{A}_{q}^{\nu,\alpha}u(x);\\
 _{x_L}D_{x}^{\alpha}u(x)&=\,_{3L}\widetilde{A}_{p,q,r,s}^{\nu,\alpha}u(x)+\mathcal{O}(h^3), ~~~~~~~~~~~
 _{3L}\widetilde{A}_{p,q,r,s}^{\alpha}u(x)=w_{p,q}^\nu\,_{2L}\widetilde{A}_{p,q}^{\nu,\alpha} u(x)+w_{r,s}^\nu\,_{2L}\widetilde{A}_{p,q}^{\nu,\alpha} u(x);\\
 _{x_L}D_{x}^{\alpha}u(x)&=\,_{4L}\widetilde{A}_{p,q,r,s,\overline{p},\overline{q},\overline{r},\overline{s}}^{\nu,\alpha}u(x)+\mathcal{O}(h^4); \\
  \end{split}
\end{equation}
where
$$
_{4L}\widetilde{A}_{p,q,r,s,\overline{p},\overline{q},\overline{r},\overline{s}}^{\nu,\alpha}u(x)
=w^\nu_{p,q,r,s}\,_{3L}\widetilde{A}_{p,q,r,s}^{\nu,\alpha} u(x)
+w^\nu_{\overline{p},\overline{q},\overline{r},\overline{s}}\,_{3L}\widetilde{A}_{\overline{p},\overline{q},\overline{r},\overline{s}}^{\nu,\alpha} u(x).
$$
Denoting  $x_i=x_L+ih$, $i=-m,\ldots, 0,1,\ldots,N_x-1,N_x,\ldots,N_x+m$,
and $h=(x_R-x_L)/N_x$ being the uniform spacestep, it can be note that
 \begin{equation*}
u(x_i)= 0, \quad {\rm for} \quad i=-m,-m+1,\ldots,0 \quad \mbox{and}\quad   i=N_x, N_x+1,\ldots, N_x+m,
\end{equation*}
where
\begin{equation}\label{2.27}
  m=\max({\rm abs}(p,q,r,s,\overline{p},\overline{q},\overline{r},\overline{s})).
\end{equation}
 Then the approximation operator of (\ref{2.25}) can be  described as
 \begin{equation}\label{2.28}
  \begin{split}
&_{1L}\widetilde{A}_{p}^{\nu,\alpha}u(x_i)=\frac{1}{h^{\alpha}}\sum_{k=0}^{i+p}l_k^{\nu,\alpha}u(x_{i-k+p})
=\frac{1}{h^{\alpha}}\sum_{k=m-p}^{i+m}l_{k+p-m}^{\nu,\alpha}u(x_{i-k+m})
=\frac{1}{h^{\alpha}}\sum_{k=0}^{i+m}l_{k+p-m}^{\nu,\alpha}u(x_{i-k+m}),
\end{split}
\end{equation}
where $l_{k+p-m}^{\nu,\alpha}=0$, when $k+p-m<0$, and $p$ is an integer.
Then
\begin{equation}\label{2.29}
\begin{split}
_{x_L}D_{x}^{\alpha}u(x_i)&=\,_{1L}\widetilde{A}_{p}^{\nu,\alpha}u(x_i)+\mathcal{O}(h)
=\frac{1}{h^{\alpha}}\sum_{k=0}^{i+m}l_{k+p-m}^{\nu,\alpha}u(x_{i-k+m})+\mathcal{O}(h), ~~~\,~~p\neq 0;\\
_{x_L}D_{x}^{\alpha}u(x_i)&=\,_{1L}\widetilde{A}_{p}^{\nu,\alpha}u(x_i)+\mathcal{O}(h^\nu)
=\frac{1}{h^{\alpha}}\sum_{k=0}^{i+m}l_{k+p-m}^{\nu,\alpha}u(x_{i-k+m})+\mathcal{O}(h^\nu), ~~p= 0;\\
_{x_L}D_{x}^{\alpha}u(x_i)&= \,_{2L}\widetilde{A}_{p,q}^{\nu,\alpha}u(x_i) +\mathcal{O}(h^2)
=\frac{1}{h^{\alpha}}\sum_{k=0}^{i+m}(w^\nu_pl_{k+p-m}^{\nu,\alpha}+w^\nu_ql_{k+q-m}^{\nu,\alpha})u(x_{i-k+m}) +\mathcal{O}(h^2);\\
_{x_L}D_{x}^{\alpha}u(x_i)&=\,_{3L}\widetilde{A}_{p,q,r,s}^{\nu,\alpha}u(x)+\mathcal{O}(h^3)
=\frac{1}{h^{\alpha}}\sum_{k=0}^{i+m}(w^\nu_{p,q}w^\nu_pl_{k+p-m}^{\nu,\alpha}\!+\!w^\nu_{p,q}w^\nu_ql_{k+q-m}^{\nu,\alpha}\\
&~~~~~~~~~~~~~~~~~~~~~~~~~~~~~~~~~~~~~~~~~~~~~~~ +w^\nu_{r,s}w^\nu_rl_{k+r-m}^{\nu,\alpha}\!+\!w^\nu_{r,s}w^\nu_sl_{k+s-m}^{\nu,\alpha})u(x_{i-k+m}) +\mathcal{O}(h^3);\\
_{x_L}D_{x}^{\alpha}u(x_i)&=\,_{4L}\widetilde{A}_{p,q,r,s,\overline{p},\overline{q},\overline{r},\overline{s}}^{\nu,\alpha}u(x_i)+\mathcal{O}(h^4)
=\frac{1}{h^{\alpha}}\sum_{k=0}^{i+m}\varphi_k^{\nu,\alpha}u(x_{i-k+m})+\mathcal{O}(h^4),
 \end{split}
\end{equation}
where
\begin{equation}\label{2.30}
\begin{split}
\varphi_k^{\nu,\alpha}&=w^\nu_{p,q,r,s}  w^\nu_{p,q}w^\nu_pl_{k+p-m}^{\nu,\alpha}+ w^\nu_{p,q,r,s}w^\nu_{p,q}w^\nu_ql_{k+q-m}^{\nu,\alpha}
+w^\nu_{p,q,r,s}w^\nu_{r,s}w^\nu_rl_{k+r-m}^{\nu,\alpha}\\
&\quad +w^\nu_{p,q,r,s}w^\nu_{r,s}w^\nu_sl_{k+s-m}^{\nu,\alpha} +w^\nu_{\overline{p},\overline{q},\overline{r},\overline{s}}w^\nu_{\overline{p},\overline{q}}w^\nu_{\overline{p}}l_{k+\overline{p}-m}^{\nu,\alpha}
+w^\nu_{\overline{p},\overline{q},\overline{r},\overline{s}}w^\nu_{\overline{p},\overline{q}}w^\nu_{\overline{q}}l_{k+\overline{q}-m}^{\nu,\alpha}\\
&\quad +w^\nu_{\overline{p},\overline{q},\overline{r},\overline{s}}w^\nu_{\overline{r},\overline{s}}w^\nu_{\overline{r}}l_{k+\overline{r}-m}^{\nu,\alpha}
+w^\nu_{\overline{p},\overline{q},\overline{r},\overline{s}}w^\nu_{\overline{r},\overline{s}}w^\nu_{\overline{s}}l_{k+\overline{s}-m}^{\nu,\alpha}.
\end{split}
\end{equation}
Taking $U=[u({x_0}),u({x_1}),\cdots,u({x_{N_x}})]^{\rm T}$,  then (\ref{2.28}) can be rewritten  as matrix form
\begin{equation}\label{2.31}
  \begin{split}
&_{1L}\widetilde{A}_{p}^{\nu,\alpha}U=\frac{1}{h^{\alpha}}A_{p}^{\nu,\alpha} U,
\end{split}
\end{equation}
where
\begin{equation}\label{2.32}
A_p^{\nu,\alpha}=\left [ \begin{matrix}
l_p^{\nu,\alpha}     &l_{p-1}^{\nu,\alpha}  & \cdots         &l_0^{\nu,\alpha}    &             &                  &               &        \\
l_{p+1}^{\nu,\alpha} &l_p^{\nu,\alpha}      &l_{p-1}^{\nu,\alpha} &  \cdots       & l_0^{\nu,\alpha} &                  &               &         \\
l_{p+2}^{\nu,\alpha} &l_{p+1}^{\nu,\alpha}  &l_p^{\nu,\alpha}     &l_{p-1}^{\nu,\alpha}&\cdots       &    l_0^{\nu,\alpha}   &               &          \\
\vdots          &    \ddots       &\ddots          & \ddots        & \ddots      &   \cdots         &     \ddots    &           \\
l_{n-2}^{\nu,\alpha} &\cdots           & \ddots         &l_{p+1}^{\nu,\alpha}&l_{p}^{\nu,\alpha}&  l_{p-1}^{\nu,\alpha} & \cdots        &l_0^{\nu,\alpha} \\
\vdots          &\ddots           & \cdots         & \ddots        & \ddots      &   \ddots         & \ddots        &\vdots       \\
l_{p+n-3}^{\nu,\alpha}&\cdots           &  \ddots        & \cdots        & \ddots      &  l_{p+1}^{\nu,\alpha} &l_{p}^{\nu,\alpha}  &l_{p-1}^\alpha\\
l_{p+n-2}^{\nu,\alpha}&l_{p+n-3}^{\nu,\alpha}&  \cdots        &l_{n-2}^{\nu,\alpha}& \cdots      &    l_{p+2}^\alpha&l_{p+1}^{\nu,\alpha}&l_{p}^{\nu,\alpha} \end{matrix}
 \right ],
\end{equation}
and $l_{k}^{\nu,\alpha}=0$, when $k<0$, and $p$ is an integer. From (\ref{2.26}) and (\ref{2.31}) we obtain
\begin{equation}\label{2.33}
\begin{split}
_{2L}\widetilde{A}_{p,q}^{{\nu,\alpha}}U=\frac{1}{h^{{\alpha}}}A_{p,q}^{\nu,\alpha} U,
&\quad A_{p,q}^{\nu,\alpha} =w^\nu_pA_{p}^{\nu,\alpha}+w^\nu_qA_{q}^{\nu,\alpha};        \\
_{3L}\widetilde{A}_{p,q,r,s}^{{\nu,\alpha}}U=\frac{1}{h^{{\alpha}}}A_{p,q,r,s}^{\nu,\alpha} U,
&\quad A_{p,q,r,s}^{\nu,\alpha} =w^\nu_{p,q}A^{\nu,\alpha}_{p,q}+w^\nu_{r,s}A^{\nu,\alpha}_{r,s};            \\
_{4L}\widetilde{A}_{p,q,r,s,\overline{p},\overline{q},\overline{r},\overline{s}}^{{\nu,\alpha}}U
=\frac{1}{h^{\alpha}}A_{p,q,r,s,\overline{p},\overline{q},\overline{r},\overline{s}}^{\nu,\alpha} U,
&\quad A_{p,q,r,s,\overline{p},\overline{q},\overline{r},\overline{s}}^{\nu,\alpha} =w^\nu_{p,q,r,s}A_{p,q,r,s}^{\nu,\alpha}
+w^\nu_{\overline{p},\overline{q},\overline{r},\overline{s}}A_{\overline{p},\overline{q},\overline{r},\overline{s}}^{\nu,\alpha}.
 \end{split}
\end{equation}
Similarly, for the right Riemann-Liouville fractional derivative, assume that
 \begin{equation*}
  \begin{split}
&_{1R}\widetilde{A}_{p}^{\nu,\alpha}u(x)=\frac{1}{h^{\alpha}}\sum_{k=0}^{[\frac{x_R-x}{h}]+p}l_k^{\nu,\alpha}u(x+(k-p)h),
\end{split}
\end{equation*}
then  there exists
\begin{equation*}
  \begin{split}
&_{1R}\widetilde{A}_{p}^{{\nu,\alpha}}u(x_i)=\frac{1}{h^{\alpha}}\!\!\!\sum_{k=0}^{N_x-i+p}\!\!\!l_k^{{\nu,\alpha}}u(x_{i+k-p})
=\frac{1}{h^{{\alpha}}}\!\!\!\sum_{k=m-p}^{N_x-i+m}\!\!\!l_{k+p-m}^{{\nu,\alpha}}u(x_{i+k-m})
=\frac{1}{h^{{\alpha}}}\!\!\!\sum_{k=0}^{N_x-i+m}\!\!\!l_{k+p-m}^{{\nu,\alpha}}u(x_{i+k-m}),
\end{split}
\end{equation*}
where $l_{k+p-m}^{\nu,\alpha}=0$, $k+p-m<0$, and $p$ is an integer. And the  fourth order approximation is
\begin{equation}\label{2.34}
\begin{split}
&_{x}D_{x_R}^{\alpha}u(x_i)=\,_{4R}\widetilde{A}_{p,q,r,s,\overline{p},\overline{q},\overline{r},\overline{s}}^{\nu,\alpha}u(x_i)+\mathcal{O}(h^4)
=\frac{1}{h^{\alpha}}\sum_{k=0}^{N_x-i+m}\varphi_k^{\nu,\alpha}u(x_{i+k-m})+\mathcal{O}(h^4),
 \end{split}
\end{equation}
where  $\varphi_k^{\nu,\alpha}$ is defined by (\ref{2.30}), and the matrices  forms are
\begin{equation}\label{2.35}
  \begin{split}
_{1R}\widetilde{A}_{p}^{{\nu,\alpha}}U=\frac{1}{h^{{\alpha}}}B_{p}^{\nu,\alpha} U, ~~~~ &B_{p}^{\nu,\alpha}=(A_{p}^{\nu,\alpha})^T\\
_{2R}\widetilde{A}_{p,q}^{{\nu,\alpha}}U=\frac{1}{h^{{\alpha}}}B_{p,q}^{\nu,\alpha} U, \quad
&B_{p,q}^{\nu,\alpha} =w^\nu_pB_{p}^{\nu,\alpha}+w^\nu_qB_{q}^{\nu,\alpha};\\
_{3R}\widetilde{A}_{p,q,r,s}^{{\nu,\alpha}}U=\frac{1}{h^{{\alpha}}}B_{p,q,r,s}^{\nu,\alpha} U, \quad
&B_{p,q,r,s}^{\nu,\alpha} =w^\nu_{p,q}B_{p,q}^{{\nu,\alpha}}+w^\nu_{r,s}B_{r,s}^{{\nu,\alpha}};\\
_{4R}\widetilde{A}_{p,q,r,s,\overline{p},\overline{q},\overline{r},\overline{s}}^{{\nu,\alpha}}U
=\frac{1}{h^{{\alpha}}}B_{p,q,r,s,\overline{p},\overline{q},\overline{r},\overline{s}}^{\nu,\alpha} U, \quad
&B_{p,q,r,s,\overline{p},\overline{q},\overline{r},\overline{s}}^{\nu,\alpha} =w^\nu_{p,q,r,s}B_{p,q,r,s}^{\nu,\alpha}
+w^\nu_{\overline{p},\overline{q},\overline{r},\overline{s}}B_{\overline{p},\overline{q},\overline{r},\overline{s}}^{\nu,\alpha}.
\end{split}
\end{equation}

\begin{remark}\label{remark2.1}
We want to emphasize that for $\nu=3$ and $4$ the coefficients in (\ref{2.18}) are the same, but they correspond to different discretization schemes; the same thing also happens for (\ref{2.19}), (\ref{2.22}), and (\ref{2.23}).
\end{remark}

\begin{remark}\label{remark2.2}
When $p=0$,  ${\alpha} \in(1,2)$ and $\nu=3,4$, then  $A_p^{{\nu,\alpha}}$ in (\ref{2.32}) reduces to the lower triangular matrix,
and it can be easily checked that all the eigenvalues of $A_p^{{\nu,\alpha}}$  are greater than one.
This is the reason that the scheme for time dependent problem
is unstable when directly using the $\nu$-order ($\nu=2,3,4,5$) Lubich's operator with ${\alpha} \in(1,2)$  to discretize space
fractional derivative.
\end{remark}

\begin{remark}\label{remark2.3}
If $\alpha<0$, $ \{l_k^{\nu,\alpha}\}_{k=0}^{\infty}$ ($\nu=1,2,3,4,5$) correspond to the coefficients of  the $\nu$-order  approximation of fractional integral operators.
\end{remark}

\begin{example}
To numerically verify the truncation error given in Lemma \ref{lemma2.8} ($\nu=5, p=0$) in a bounded domain.
We utilize the approximation (\ref{2.28}) to simulate the following equation
$$_{0}D_x^{\alpha}u(x)=\frac{\Gamma(9)}{\Gamma(9-\alpha)}x^{8-\alpha},~~x \in(0,1),$$
when $\alpha <0$, the fractional operator $_{0}D_x^{\alpha}$ is a fractional integral operator; if $\alpha \in (0,1)$ we take $u(0)=0$; and if $\alpha \in (1,2)$ let $u(0)=0$, $u(1)=1$; the exact solution of the above equation is $u(x)=x^8$.
\end{example}
\begin{table}[h]\fontsize{9.5pt}{12pt}\selectfont
 \begin{center}
  \caption {The maximum errors and convergent orders for  the second equation of (\ref{2.29}), when $\nu=5$, $p=0$.} \vspace{5pt}
\begin{tabular*}{\linewidth}{@{\extracolsep{\fill}}*{10}{c}}                                    \hline  
$h$& $\alpha=-0.5$&  Rate       & $\alpha=0.5$  & Rate       & $\alpha=1.8$ &   Rate    \\\hline
~~~1/10&  8.0041e-04  &             & 4.0005e-03     &            & 6.9882e-02   &        \\
~~~1/20&  4.9935e-05  &  4.0026     & 2.0652e-04     & 4.2758     & 2.8034e-03   & 4.6397   \\
~~~1/40&  2.1214e-06  &  4.5570     & 7.8935e-06     & 4.7095     & 1.2005e-04   & 4.5454   \\
~~~1/60&  3.0790e-07  &  4.7601     & 9.3316e-07     & 5.2661     & 1.3775e-05   & 5.3397   \\ \hline
    \end{tabular*}\label{example:1}
  \end{center}
\end{table}
Table \ref{example:1} numerically verifies Lemma \ref{lemma2.8}, and shows that the truncation errors are $\mathcal{O}(h^5)$.

\subsection{Effective fourth order discretization for space fractional derivatives  }
This subsection focuses on how to choose the parameters ${p,q,r,s,\overline{p},\overline{q},\overline{r},\overline{s}}$ such that all the
eigenvalues of the matrix $A_{p,q,r,s,\overline{p},\overline{q},\overline{r},\overline{s}}^{\nu,\alpha}$, ($\nu=3,4$)
have negative real parts; this means that the corresponding schemes work for space fractional derivatives. Since
$B_{p,q,r,s,\overline{p},\overline{q},\overline{r},\overline{s}}^{\nu,\alpha}$ is the transpose of
$A_{p,q,r,s,\overline{p},\overline{q},\overline{r},\overline{s}}^{\nu,\alpha}$, we don't need to discuss them separately.

\begin{lemma}\cite[p.\,28]{Quarteroni:07}\label{lemma2.9}
A real matrix $A$ of order $n$ is positive definite  if and only if  its symmetric part $H=\frac{A+A^T}{2}$ is positive definite.
Let $H \in \mathbb{R}^{n\times n}$ be symmetric, then $H$ is positive definite if and only if the eigenvalues of $H$ are positive.
\end{lemma}

\begin{lemma}\cite[p.\,184]{Quarteroni:07}\label{lemma2.10}
If $A \in \mathbb{C}^{n \times n}$, let $H=\frac{A+A^H}{2}$ be the hermitian part of $A$,  then for any eigenvalue $\lambda$ of  $A$,
the real part $\Re(\lambda(A))$ satisfies
\begin{equation*}
  \lambda_{\min}(H) \leq \Re(\lambda(A)) \leq \lambda_{\max}(H),
\end{equation*}
where $\lambda_{\min}(H)$ and $\lambda_{\max}(H)$ are the minimum and maximum of the eigenvalues of $H$, respectively.
\end{lemma}

\begin{definition}\cite[p.\,13]{Chan:07}\label{definition2.8}
Let  $n \times n$ Toeplitz  matrix  $T_n$ be of the following form:
\begin{equation*}
T_n=\left [ \begin{matrix}
                      t_0           &      t_{-1}             &      \cdots         &       t_{2-n}       &       t_{1-n}      \\
                      t_{1}         &      t_{0}              &      t_{-1}         &      \cdots         &       t_{2-n}        \\
                     \vdots         &      t_{1}              &      t_{0}          &      \ddots         &        \vdots            \\
                     t_{n-2}        &      \cdots             &      \ddots         &      \ddots         &        t_{-1}    \\
                     t_{n-1}        &       t_{n-2}           &      \cdots         &       t_1           &        t_0
 \end{matrix}
 \right ];
\end{equation*}
i.e., $t_{i,j}=t_{i-j}$ and $T_n$ is constant along its diagonals. Assume that the diagonals $\{t_k\}_{k=-n+1}^{n-1}$ are the Fourier coefficients of a function
$f$, i.e.,
\begin{equation*}
  t_k=\frac{1}{2\pi}\int_{-\pi}^{\pi}f(x)e^{-ikx}dx,
\end{equation*}
then the function $f$ is called the generating function of $T_n$.
\end{definition}

\begin{lemma}\cite[p.\,13-15]{Chan:07}\label{lemma2.11} (Grenander-Szeg\"{o} theorem) Let $T_n$ be given by above matrix with a generating function $f$,
where $f$ is a $2\pi$-periodic continuous real-valued functions defined on $[-\pi,\pi]$.
Let $\lambda_{\min}(T_n)$ and $\lambda_{\max}(T_n)$ denote the smallest and largest eigenvalues of $T_n$, respectively. Then we have
\begin{equation*}
  f_{\min} \leq \lambda_{\min}(T_n) \leq \lambda_{\max}(T_n) \leq f_{\max},
\end{equation*}
where $f_{\min}$ and  $f_{\max}$  is the minimum and maximum values of $f(x)$, respectively.
Moreover, if $f_{\min}< f_{\max}$, then all eigenvalues of $T_n$ satisfies
\begin{equation*}
  f_{\min} < \lambda(T_n) < f_{\max},
\end{equation*}
for all $n>0$; In particular, if  $f_{\min}>0$, then $T_n$ is positive definite.
\end{lemma}


\begin{lemma}\label{lemma2.12}
Let $A_{p,q}^{\nu,\alpha} $ be given in (\ref{2.33}),
 $(p,q)=(1,q)$, $q$ is an integer ($q\neq 1$), $1<\alpha<2$ and denote that
\begin{equation}\label{2.37}
H^\nu_{p,q}=\frac{A_{p,q}^{\nu,\alpha}+\left(A_{p,q}^{\nu,\alpha}\right)^T}{2},~~\nu=3,4,
\end{equation}
then the generating function of $H_{p,q}^{\nu}$ is  
\begin{equation}\label{2.36}
\begin{split}
f^\nu_{p,q}(\alpha,x)&=\left(2sin\frac{x}{2}\right)^{\alpha}\left( a_\nu^2+b_\nu^2\right)^{\frac{\alpha}{2}}
 \cdot \left[w_p^\nu\cos\left (\alpha(\frac{x}{2}-\frac{\pi}{2}+\theta)-x\right )+w_q^\nu\cos\left (\alpha(\frac{x}{2}-\frac{\pi}{2}+\theta)-qx\right)\right],
\end{split}
\end{equation}
where $w_p^\nu$ and $w_q^\nu$ are defined in (\ref{2.18}), $\theta=-\arctan\frac{b_\nu}{a_\nu}$,
$$a_3=\frac{1}{6}\left(11-7\cos(x)+2\cos(2x)\right),~~b_3=\frac{1}{6}\left(7\sin(x)-2\sin(2x)\right);$$
and
$$a_4=\frac{1}{12}\left(25-23\cos(x)+13\cos(2x)-3\cos(3x)\right),~~b_4=\frac{1}{12}\left(23\sin(x)-13\sin(2x)+3\sin(3x)\right).$$

\end{lemma}
\begin{proof}
\begin{description}
 \item[(1)] For $(p,q)=(1,q)$,  $q\leq 0$, we have  $A_{p,q}^{\nu,\alpha} =w^\nu_pA_{p}^{\nu,\alpha}+w^\nu_qA_{q}^{\nu,\alpha}$ and
\end{description}
\begin{equation*}
A_{p,q}^{\nu,\alpha}=\left [ \begin{matrix}
\phi_1^{{\nu,\alpha}}             &\phi_0^{{\nu,\alpha}}         &                          &                          &                             \\
\phi_2^{{\nu,\alpha}}             &\phi_1^{{\nu,\alpha}}         &\phi_0^{{\nu,\alpha}}           &                          &                            \\
\vdots                      &\ddots                  &       \ddots             &        \ddots            &     \\
\phi_{N_x-2}^{{\nu,\alpha}}       &\ddots                  &        \ddots            &   \phi_1^{{\nu,\alpha}}        & \phi_0^{{\nu,\alpha}} \\
\phi_{N_x-1}^{{\nu,\alpha}}       &\phi_{N_x-2}^{{\nu,\alpha}}   &         \cdots           &\phi_2^{{\nu,\alpha}}           & \phi_1^{{\nu,\alpha}}
 \end{matrix}
 \right ],
\end{equation*}
with
 \begin{equation*}
\phi_k^{{\nu,\alpha}}=\left\{ \begin{array}
 {l@{\quad} l}
w^\nu_pl_k^{\nu,\alpha},&0\leq k\leq -q,\\\\
w^\nu_pl_k^{\nu,\alpha}+w^\nu_ql_{k+q-1}^{\nu,\alpha},&k>-q.
 \end{array}
 \right.
\end{equation*}
The generating functions of $A_{p,q}^{\nu,\alpha}$ and $(A_{p,q}^{{\nu,\alpha}})^T$ are
\begin{equation*}
  f_{A_{p,q}^{\nu,\alpha}}(x)=\sum_{k=0}^{\infty}\phi_k^{{\nu,\alpha}}e^{i(k-1)x},~~\mbox{and}~~ f_{(A_{p,q}^{\nu,\alpha})^T}(x)=\sum_{k=0}^{\infty}\phi_k^{{\nu,\alpha}}e^{-i(k-1)x},
\end{equation*}
respectively.   Then
 $$f^\nu_{p,q}(\alpha,x)=\frac{f_{A_{p,q}^{\nu,\alpha}}(x)\,+\,f_{(A_{p,q}^{\nu,\alpha})^T}(x)}{2},$$
 is the generating function of $H^\nu_{p,q}$.
Since $f_{A_{p,q}^{\nu,\alpha}}(x)$ and $f_{(A_{p,q}^{\nu,\alpha})^T}(x)$ are mutually conjugated, then  $f^\nu_{p,q}({\alpha},x)$ is a $2\pi$-periodic continuous real-valued functions defined on $[-\pi,\pi]$. Moreover, $f^\nu_{p,q}({\alpha},x)$ is an even function, so we just need to  consider its principal value on $[0,\pi]$. According to the following equations
\begin{equation*}
  (1-e^{\pm ix})^{\alpha}=\left(2sin\frac{x}{2}\right)^{\alpha}e^{\pm i\alpha(\frac{x}{2}-\frac{\pi}{2})},
  ~~\left(a_\nu - b_\nu i\right)^{\alpha}=\left( a_\nu^2+b_\nu^2\right)^{\frac{\alpha}{2}} e^{i\alpha \theta},
\end{equation*}
where $\nu=3,4$ and $\theta=-\arctan\frac{b_\nu}{a_\nu}$,  then we have
\begin{equation*}
\begin{split}
  f_{A_{p,q}^{3,\alpha}}(x)&=\sum_{k=0}^{\infty}\phi_k^{{3,\alpha}}e^{i(k-1)x}
  =w_p^3 e^{-ix}\sum_{k=0}^{\infty}l_k^{{3,\alpha}}e^{ikx} +w_q^3 e^{-iqx}\sum_{k=0}^{\infty}l_k^{{3,\alpha}}e^{ikx}\\
  &=w_p^3 e^{-ix}(1-e^{ix})^\alpha \left( \frac{11}{6}-\frac{7}{6}e^{ix}+\frac{1}{3}e^{2ix} \right)^\alpha
    +w_q^3 e^{-iqx}(1-e^{ix})^\alpha \left( \frac{11}{6}-\frac{7}{6}e^{ix}+\frac{1}{3}e^{2ix} \right)^\alpha\\
    &=w_p^3 e^{-ix}(1-e^{ix})^\alpha \left( a_3-b_3i \right)^\alpha
    +w_q^3 e^{-iqx}(1-e^{ix})^\alpha \left( a_3-b_3i \right)^\alpha\\
&=w_p^3 e^{-ix}\left(2sin\frac{x}{2}\right)^{\alpha}e^{i\alpha(\frac{x}{2}-\frac{\pi}{2})} \left( a_3^2+b_3^2\right)^{\frac{\alpha}{2}} e^{i\alpha \theta}
+w_q^3 e^{-iqx}\left(2sin\frac{x}{2}\right)^{\alpha}e^{i\alpha(\frac{x}{2}-\frac{\pi}{2})} \left( a_3^2+b_3^2\right)^{\frac{\alpha}{2}} e^{i\alpha \theta}\\
&=\left(2sin\frac{x}{2}\right)^{\alpha} \left( a_3^2+b_3^2\right)^{\frac{\alpha}{2}}
\Big\{ w_p^3  e^{i \left[ \left (\alpha(\frac{x}{2}-\frac{\pi}{2}+\theta)-x\right ) \right] }
+w_q^3 e^{i \left[ \left (\alpha(\frac{x}{2}-\frac{\pi}{2}+\theta)-qx\right ) \right] }\Big \};\\
  \end{split}
\end{equation*}
and
\begin{equation*}
\begin{split}
  f_{A_{p,q}^{4,\alpha}}(x)&=\sum_{k=0}^{\infty}\phi_k^{{4,\alpha}}e^{i(k-1)x}
  =w_p^4 e^{-ix}\sum_{k=0}^{\infty}l_k^{{4,\alpha}}e^{ikx} +w_q^4 e^{-iqx}\sum_{k=0}^{\infty}l_k^{{4,\alpha}}e^{ikx}\\
  &=w_p^4 e^{-ix}(1-e^{ix})^\alpha \left( \frac{25}{12}-\frac{23}{12}e^{ix}+\frac{13}{12}e^{2ix}-\frac{1}{4}e^{3ix} \right)^\alpha\\
    &\quad +w_q^4 e^{-iqx}(1-e^{ix})^\alpha\left( \frac{25}{12}-\frac{23}{12}e^{ix}+\frac{13}{12}e^{2ix}-\frac{1}{4}e^{3ix} \right)^\alpha\\
    &=w_p^4 e^{-ix}(1-e^{ix})^\alpha \left( a_4-b_4i \right)^\alpha
    +w_q^4 e^{-iqx}(1-e^{ix})^\alpha \left( a_4-b_4i \right)^\alpha\\
&=w_p^4 e^{-ix}\left(2sin\frac{x}{2}\right)^{\alpha}e^{i\alpha(\frac{x}{2}-\frac{\pi}{2})} \left( a_4^2+b_4^2\right)^{\frac{\alpha}{2}} e^{i\alpha \theta}
+w_q^4 e^{-iqx}\left(2sin\frac{x}{2}\right)^{\alpha}e^{i\alpha(\frac{x}{2}-\frac{\pi}{2})} \left( a_4^2+b_4^2\right)^{\frac{\alpha}{2}} e^{i\alpha \theta}\\
&=\left(2sin\frac{x}{2}\right)^{\alpha} \left( a_4^2+b_4^2\right)^{\frac{\alpha}{2}}
\Big\{ w_p^4  e^{i \left[ \left (\alpha(\frac{x}{2}-\frac{\pi}{2}+\theta)-x\right ) \right] }
+w_q^4 e^{i \left[ \left (\alpha(\frac{x}{2}-\frac{\pi}{2}+\theta)-qx\right ) \right] }\Big \}.\\
  \end{split}
\end{equation*}
By the similar way, there exists 
\begin{equation*}
\begin{split}
 f_{(A_{p,q}^{\nu,\alpha})^T}(x)
=\left(2sin\frac{x}{2}\right)^{\alpha} \left( a_\nu^2+b_\nu^2\right)^{\frac{\alpha}{2}}
\Big\{ w_p^\nu  e^{-i \left[ \left (\alpha(\frac{x}{2}-\frac{\pi}{2}+\theta)-x\right ) \right] }
+w_q^\nu e^{-i \left[ \left (\alpha(\frac{x}{2}-\frac{\pi}{2}+\theta)-qx\right ) \right] }\Big \}.
  \end{split}
\end{equation*}
Therefore, it is easy to get
\begin{equation*}
\begin{split}
f^\nu_{p,q}(\alpha,x)&=\frac{1}{2}\left(\sum_{k=0}^{\infty}\phi_k^{\nu,\alpha}e^{i(k-1)x}+\sum_{k=0}^{\infty}\phi_k^{\nu,\alpha}e^{-i(k-1)x}\right)\\
&=\left(2sin\frac{x}{2}\right)^{\alpha}\left( a_\nu^2+b_\nu^2\right)^{\frac{\alpha}{2}}
 \cdot \left[w_p^\nu\cos\left (\alpha(\frac{x}{2}-\frac{\pi}{2}+\theta)-x\right )+w_q^\nu\cos\left (\alpha(\frac{x}{2}-\frac{\pi}{2}+\theta)-qx\right)\right].
\end{split}
\end{equation*}
\begin{description}
 \item[(2)] For $(p,q)=(1,q)$,  $q\geq 2$, we have  $A_{p,q}^{\nu,\alpha} =w^\nu_pA_{p}^{\nu,\alpha}+w^\nu_qA_{q}^{\nu,\alpha}$ and
\end{description}

\begin{equation*}
A_{p,q}^{\nu,\alpha}=\left [ \begin{matrix}
\phi_p^{\nu,\alpha}      &\phi_{p-1}^{\nu,\alpha}   & \cdots          &\phi_0^{\nu,\alpha}     &               &                  &                 &        \\
\phi_{p+1}^{\nu,\alpha} &\phi_p^{\nu,\alpha}   &\phi_{p-1}^{\nu,\alpha}&  \cdots          & \phi_0^{\nu,\alpha} &                  &                 &         \\
\phi_{p+2}^{\nu,\alpha}  &\phi_{p+1}^{\nu,\alpha}   &\phi_p^{\nu,\alpha}    &\phi_{p-1}^{\nu,\alpha} &\cdots         &\phi_0^{\nu,\alpha}     &                 &        \\
\vdots             &    \ddots          &\ddots           & \ddots           & \ddots        &   \cdots         & \ddots          &             \\
\phi_{n-2}^{\nu,\alpha}  &\cdots              & \ddots          &\phi_{p+1}^{\nu,\alpha} &\phi_{p}^{\nu,\alpha}&\phi_{p-1}^{\nu,\alpha} & \cdots          &\phi_0^{\nu,\alpha} \\
\vdots             &\ddots              & \cdots          & \ddots           & \ddots        &   \ddots         & \ddots          &\vdots  \\
\phi_{p+n-3}^{\nu,\alpha}&\cdots              &  \ddots         & \cdots           & \ddots        &\phi_{p+1}^{\nu,\alpha} &\phi_{p}^{\nu,\alpha}  &\phi_{p-1}^{\nu,\alpha} \\
\phi_{p+n-2}^{\nu,\alpha}&\phi_{p+n-3}^{\nu,\alpha}&\cdots&\phi_{n-2}^{\nu,\alpha}& \cdots &\phi_{p+2}^{\nu,\alpha}&\phi_{p+1}^{\nu,\alpha}&\phi_{p}^{\nu,\alpha}
 \end{matrix}
 \right ],
\end{equation*}
with
 \begin{equation*}
\phi_k^{{\nu,\alpha}}=\left\{ \begin{array}
 {l@{\quad} l}
w^\nu_ql_k^{{\nu,\alpha}},&0\leq k\leq q-2,\\
 w^\nu_ql_{k}^{{\nu,\alpha}} +w^\nu_pl_{k-q+1}^{{\nu,\alpha}},&k>q-2.
 \end{array}
 \right.
\end{equation*}
The generating functions of $A_{p,q}^{\nu,\alpha}$ and $(A_{p,q}^{{\nu,\alpha}})^T$ are
\begin{equation*}
  f_{A_{p,q}^{\nu,\alpha}}(x)=\sum_{k=0}^{\infty}\phi_k^{{\nu,\alpha}}e^{i(k-q)x},~~\mbox{and}~~ f_{(A_{p,q}^{\nu,\alpha})^T}(x)=\sum_{k=0}^{\infty}\phi_k^{{\nu,\alpha}}e^{-i(k-q)x},
\end{equation*}
respectively.  Then
 $$f^\nu_{p,q}(\alpha,x)=\frac{f_{A_{p,q}^{\nu,\alpha}}(x)\,+\,f_{(A_{p,q}^{\nu,\alpha})^T}(x)}{2},$$
 is the generating function of $H^\nu_{p,q}$.

By the similar way, for $q\geq 2$, their exists
\begin{equation*}
\begin{split}
f^\nu_{p,q}(\alpha,x)&=\left(2sin\frac{x}{2}\right)^{\alpha}\left( a_\nu^2+b_\nu^2\right)^{\frac{\alpha}{2}}
 \cdot \left[w_p^\nu\cos\left (\alpha(\frac{x}{2}-\frac{\pi}{2}+\theta)-x\right )+w_q^\nu\cos\left (\alpha(\frac{x}{2}-\frac{\pi}{2}+\theta)-qx\right)\right].
\end{split}
\end{equation*}
It can be noted that $f^\nu_{p,q}(\alpha,x)$ has the same form when $q\leq 0$ and $q \geq 2$, $p=1$, Then $H_{p,q}^{\nu}$ has the generating function
\begin{equation*}
\begin{split}
f^\nu_{p,q}(\alpha,x)&=\left(2sin\frac{x}{2}\right)^{\alpha}\left( a_\nu^2+b_\nu^2\right)^{\frac{\alpha}{2}}
 \cdot \left[w_p^\nu\cos\left (\alpha(\frac{x}{2}-\frac{\pi}{2}+\theta)-x\right )+w_q^\nu\cos\left (\alpha(\frac{x}{2}-\frac{\pi}{2}+\theta)-qx\right)\right].
\end{split}
\end{equation*}
\end{proof}

\begin{theorem}\label{theorem2.10} (Case $\nu=3,4$; Effective 4th order schemes)
Let $A_{\nu,\alpha}=:A_{p,q,r,s,\overline{p},\overline{q},\overline{r},\overline{s}}^{\nu,\alpha}$
 with  $1<\alpha<2$,  be given in (\ref{2.33}) and $ (p,q,r,s,\overline{p},\overline{q},\overline{r},\overline{s})=(1,-1,1,2,1,-1,1,3)$.
Then any eigenvalue $\lambda$ of $A_{\nu,\alpha}$ satisfies
$$
\Re(\lambda(A_{\nu,\alpha}))<0,~~\nu=3,4,
$$
and  the matrices $A_{\nu,\alpha}$ and $A_{\nu,\alpha}^T$ are negative definite.
\end{theorem}
\begin{proof}
   Take
$$
 H^\nu_{\alpha}=\frac{A_{\nu,\alpha}+A_{\nu,\alpha}^T}{2}
 =w^\nu_{p,q,r,s}w^\nu_{p,q}H^\nu_{p,q}+w^\nu_{p,q,r,s}w^\nu_{r,s}H^\nu_{r,s}
 +w^\nu_{\overline{p},\overline{q},\overline{r},\overline{s}}w^\nu_{\overline{p},\overline{q}}H^\nu_{\overline{p},\overline{q}}
 +w^\nu_{\overline{p},\overline{q},\overline{r},\overline{s}}w^\nu_{\overline{r},\overline{s}}H^\nu_{\overline{r},\overline{s}},
$$
 where $H^\nu_{p,q}$, $H^\nu_{r,s}$, $H^\nu_{\overline{p},\overline{q}}$ and $H^\nu_{\overline{r},\overline{s}}$ are defined in (\ref{2.37}).  Then
 \begin{equation*}
 \begin{split}
   f^\nu(\alpha,x)=&w^\nu_{p,q,r,s}w^\nu_{p,q}f^\nu_{p,q}(\alpha,x)+w^\nu_{p,q,r,s}w^\nu_{r,s}f^\nu_{r,s}(\alpha,x)+w^\nu_{\overline{p},\overline{q},\overline{r},\overline{s}}w^\nu_{\overline{p},\overline{q}}f^\nu_{\overline{p},\overline{q}}(\alpha,x)
 +w^\nu_{\overline{p},\overline{q},\overline{r},\overline{s}}w^\nu_{\overline{r},\overline{s}}f^\nu_{\overline{r},\overline{s}}(\alpha,x),
 \end{split}
 \end{equation*}
 is the generating function of $H^\nu_\alpha $,
 where $f^\nu_{p,q}(\alpha,x)$, $f^\nu_{r,s}(\alpha,x)$, $f^\nu_{\overline{p},\overline{q}}(\alpha,x)$ and $f^\nu_{\overline{r},\overline{s}}(\alpha,x)$ are given by (\ref{2.36}).   And there exists $f^\nu(\alpha,x) \leq 0$  (see Figs. \ref{FIG.1}-\ref{FIG.2}), when $ (p,q,r,s,\overline{p},\overline{q},\overline{r},\overline{s})=(1,-1,1,2,1,-1,1,3)$.

Since $f^\nu(\alpha,x)$ is not identically zero for $\alpha \in (1,2)$,   from Lemma \ref{lemma2.11},
it implies that $\lambda (H^\nu_\alpha)<0$, and $H^\nu_\alpha$ is negative definite.
Then we obtain  $\Re(\lambda(A_{\nu,\alpha}))<0$ by Lemma  \ref{lemma2.10},
and the matrices $A_{\nu,\alpha}$ and $A_{\nu,\alpha}^T$ are negative definite from  Lemma \ref{lemma2.9}.
\end{proof}


\begin{figure}[t]
    \begin{minipage}[t]{0.45\linewidth}
    \includegraphics[scale=0.47]{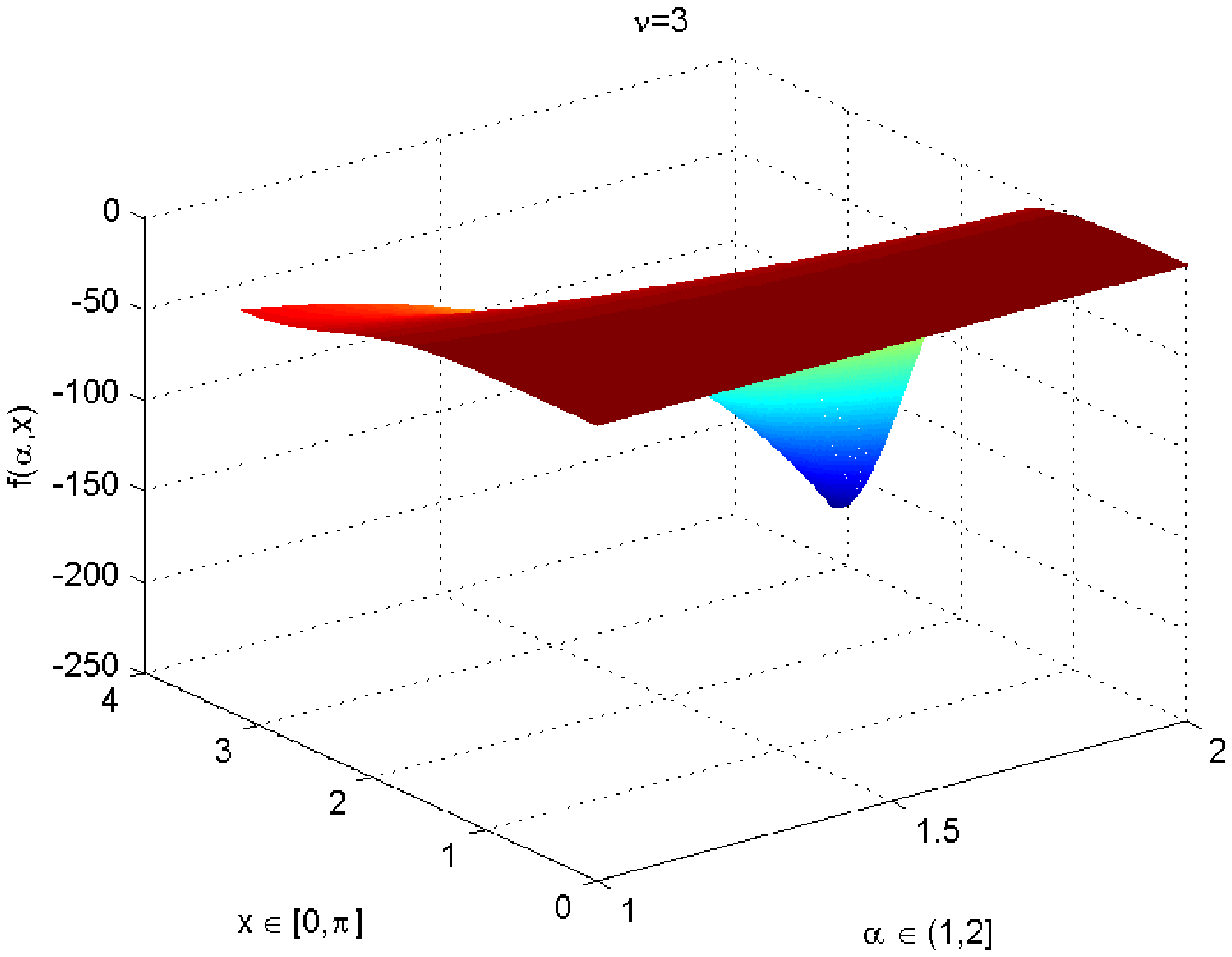}
    \caption{$f^\nu(\alpha,x)$ for $\nu=3$}  \label{FIG.1}
    \end{minipage}
    \begin{minipage}[t]{0.45\linewidth}
    \includegraphics[scale=0.47]{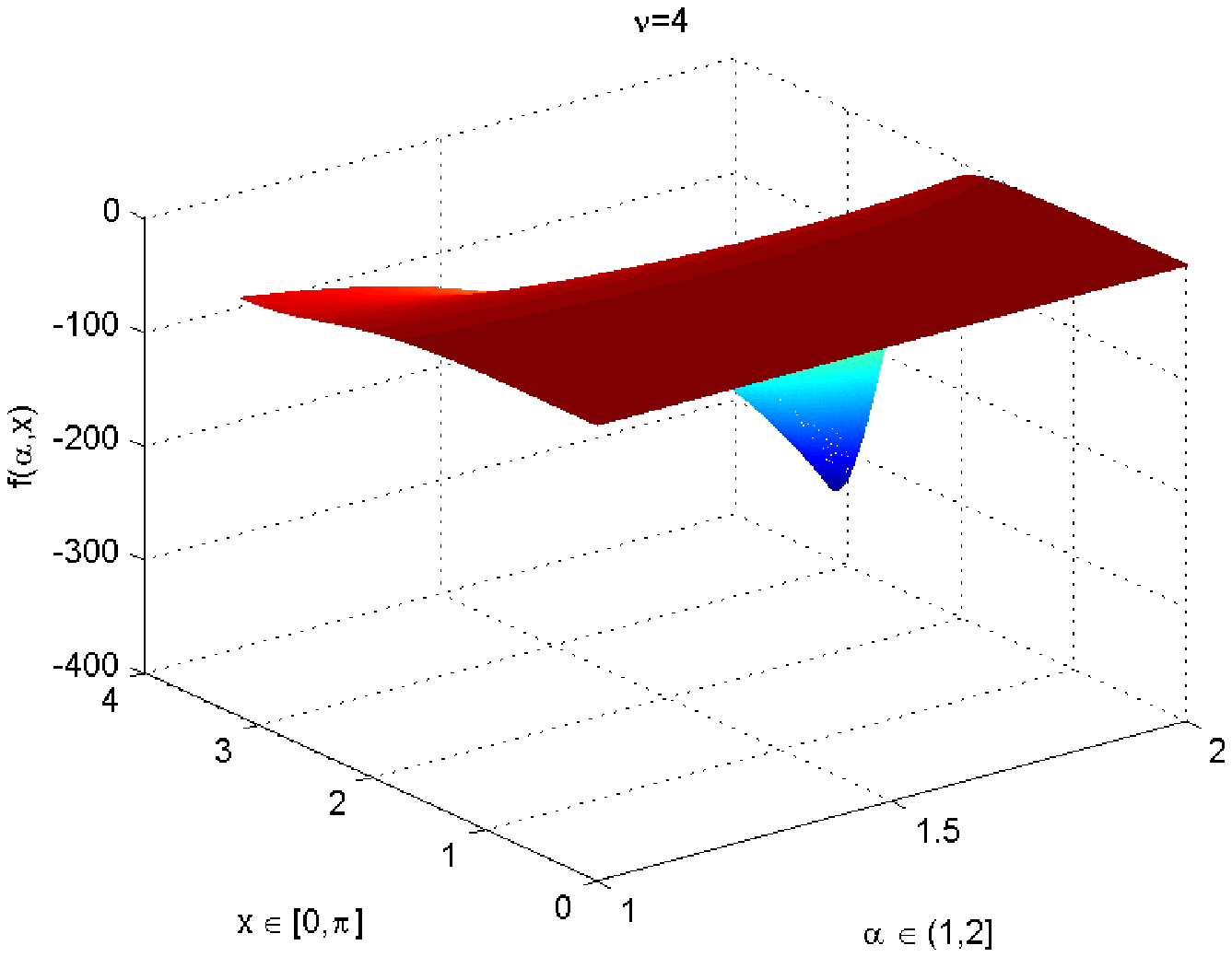}
    \caption{$f^\nu(\alpha,x)$ for $\nu=4$}  \label{FIG.2}
    \end{minipage}
\end{figure}


\section{Simple application to space fractional diffusion equation }

Similar to the discussions in [\cite{Chen:13}],  in this section,
we apply the 4th order discretizations to solve the following fractional diffusion equation with variable coefficients
\begin{equation}\label{3.1}
\frac{\partial u(x,t) }{\partial t}=d_{+}(x) _{x_L}D_x^{\alpha}u(x,t)+d_{-}(x)  _{x}D_{x_R}^{\alpha}u(x,t) +f(x,t).
\end{equation}
In the time direction, the Crank-Nicolson scheme is used.
 The 4th order left fractional approximation operator
(\ref{2.29}), and right fractional approximation operator
(\ref{2.34}) are respectively used to discretize the left Riemann-Liouville fractional
derivative, and right Riemann-Liouville fractional derivative.

Let the mesh points $x_i=x_L+ih$, $i=-m,\ldots, 0,1,\ldots,N_x-1,N_x,\ldots,N_x+m$, with $m$ in (\ref{2.27})
and $t_n=n\tau$, $0\leq n \leq {N_t}$, where $h=(x_R-x_L)/{N_x}$, $\tau=T/{N_t}$,
i.e., $h$ is the uniform space step size and $\tau$ the time step size.
Taking  $u_{i}^n$ as the approximated value of $u(x_i,t_n)$ and  $d_{+,i}=d_{+}(x_i)$, $d_{-,i}=d_{-}(x_i)$,
$f_i^{n+1/2}=f(x_i,t_{n+1/2})$, where $t_{n+1/2}=(t_n+t_{n+1})/2$.

The full discretization of (\ref{3.1}) has the following form [\cite{Chen:13}]
 \begin{equation}\label{3.2}
\begin{split}
 &  \left[I - \frac{\tau }{2h^{{\alpha}}}\left(  D_{+}A_{{\nu,\alpha}}+D_{-}A_{{\nu,\alpha}}^T \right) \right] U^{n+1}
= \left[I+  \frac{\tau }{2h^{{\alpha}}}\left(  D_{+}A_{{\nu,\alpha}}+D_{-}A_{{\nu,\alpha}}^T  \right) \right] U^{n}+\tau F^{n+1/2},
\end{split}
\end{equation}
where  $A_{\nu,\alpha}=:A_{p,q,r,s,\overline{p},\overline{q},\overline{r},\overline{s}}^{\nu,\alpha}$
 with  $1<\alpha<2$, $\nu=3,4$, are given in (\ref{2.33}), and
\begin{equation}\label{3.3}
D_{+}=\left [ \begin{matrix}
  d_{+,0}                        \\
       & d_{+,1}                  \\
       &       & \ddots            \\
       &       &      &d_{+,N_x}
 \end{matrix}
 \right ],
 \quad
D_{-}=\left [ \begin{matrix}
  d_{-,0}                        \\
       & d_{-,1}                  \\
       &       & \ddots            \\
       &       &      &d_{-,N_x}
 \end{matrix}
 \right ],
\end{equation}
and
 \begin{equation*}
 U^{n}=[u_0^n,u_1^n,\ldots,u_{N_x}^n]^{\rm T}, ~~F^{n+1/2}=[f_0^{n+1/2},f_1^{n+1/2},\ldots,f_{N_x}^{n+1/2}]^{\rm T}.
  \end{equation*}

By the same way given in [\cite{Chen:13}], we can theoretically prove that the difference scheme is unconditionally stable and 4th order
convergent in space directions and 2nd order convergent in time direction; the proofs are omitted here.

\begin{theorem}(\cite{Chen:13})
Let $D_{-}=\kappa_\alpha D_{+}$,  where $\kappa_\alpha$ is any
given nonnegative constant,
then the difference scheme (\ref{3.2})  with $\alpha \in (1,2)$
is unconditionally stable.
\end{theorem}

\begin{theorem}(\cite{Chen:13})
Let $u(x_i,t_n)$ be the exact solution of (\ref{3.1})  with $\alpha \in (1,2)$, and  $u_i^n$ be the  solution of
the  finite difference scheme (\ref{3.2}).
If $D_{-}=\kappa_\alpha D_{+}$,  then there is a positive constant $C$ such that
\begin{equation*}
  \begin{split}
||u(x_i,t_n)-u_i^n|| \leq  C (\tau^2+h^4), \quad i=0,2,\ldots,N_x;\,n=0,1,\ldots,N_t,
  \end{split}
  \end{equation*}
where $\kappa_\alpha$ is any
given nonnegative constant.
\end{theorem}

\section{Numerical Results}

In this section, we numerically verify the above theoretical results including convergence
rates and numerical stability.  And the $ l_\infty$ norm is used to measure the numerical errors.

Consider the fractional diffusion
equation (\ref{3.1}) [\cite{Chen:13}] in the domain $0< x < 2 $, $0 < t \leq 1$, with the variable  coefficients $d_{+}(x)=x^\alpha$, $d_{-}(x)=2x^\alpha$, the forcing function
\begin{equation*}
\begin{split}
 f(x,t)=&cos(t+1)x^4(2-x)^4 -x^\alpha sin(t+1) \Big[\frac{\Gamma(9)}{\Gamma(9-\alpha)}(x^{8-\alpha}+2(2-x)^{8-\alpha})\\
 &-8\frac{\Gamma(8)}{\Gamma(8-\alpha)}(x^{7-\alpha}+2(2-x)^{7-\alpha})
 +24\frac{\Gamma(7)}{\Gamma(7-\alpha)}(x^{6-\alpha}+2(2-x)^{6-\alpha})\\
 &-32\frac{\Gamma(6)}{\Gamma(6-\alpha)}(x^{5-\alpha}+2(2-x)^{5-\alpha})
 +16\frac{\Gamma(5)}{\Gamma(5-\alpha)}(x^{4-\alpha}+2(2-x)^{4-\alpha}) \Big],
  \end{split}
\end{equation*}
and the initial condition $u(x,0)=sin(1)x^4(2-x)^4$, the boundary conditions $u(0,t)=u(1,t)=0$,
and the exact solution of the equation is $u(x,t)=sin(t+1)x^4(2-x)^4.$

\begin{table}[h]\fontsize{9.5pt}{12pt}\selectfont
 \begin{center}
  \caption {The maximum errors and convergent orders for the scheme (\ref{3.2})
  of the fractional  diffusion equation (\ref{3.1}) at t=1 and $\tau=h^2$, where $ (p,q,r,s,\overline{p},\overline{q},\overline{r},\overline{s})=(1,-1,1,2,1,-1,1,3)$,  $\nu=3,4$. } \vspace{5pt}
\begin{tabular*}{\linewidth}{@{\extracolsep{\fill}}*{8}{c}}                                    \hline  
$\nu=3,h$& $\alpha=1.1$ &  Rate       & $\alpha=1.5$    & Rate       & $\alpha=1.8$ &   Rate   \\\hline
~~~1/10&  1.8349e-02  &             & 2.1073e-02      &            & 2.3337e-02   &           \\
~~~1/20&  1.4015e-03  &  3.7107     & 1.8381e-03      & 3.5191     & 2.3106e-03   & 3.3362     \\
~~~1/40&  8.8517e-05  &  3.9849     & 1.2004e-04      & 3.9367     & 1.6131e-04   & 3.8404      \\
~~~1/80&  5.2342e-06  &  4.0799     & 7.5382e-06      & 3.9931     & 1.0478e-05   & 3.9443       \\ \hline
$\nu=4,h$& $\alpha=1.1$ &  Rate       & $\alpha=1.5$    & Rate       & $\alpha=1.8$ &   Rate        \\\hline
~~~1/10&  1.7241e-02  &             & 9.6037e-03      &            & 5.8735e-03   &                \\
~~~1/20&  7.9269e-04  &  4.4430     & 5.2600e-04      & 4.1905     & 3.4793e-04   & 4.0774          \\
~~~1/40&  3.4558e-05  &  4.5197     & 2.4926e-05      & 4.3994     & 2.1158e-05   & 4.0395           \\
~~~1/80&  1.4824e-06  &  4.5430     & 1.1512e-06      & 4.4364     &1.3045e-06    & 4.0196            \\ \hline
    \end{tabular*}\label{tab:2}
  \end{center}
\end{table}
 Table \ref{tab:2} shows the maximum errors at time $t=1$, and the time and space stepsizes are taken as
$\tau=h^2$. The numerical results confirm the $\mathcal{O}(\tau^2+h^4)$ order convergence.

\section{Conclusions}
For solving classical differential equations, usually people think that the high order schemes can reduce computational cost for getting some requested accuracy, i.e., they think that the improved accuracy can overtake the increased workload. For the issue of computational cost of fractional differential equations, the high order schemes play more fundamental role; since they can greatly increase the accuracy but without adding new cost. As the sequel of [\cite{Chen:13}], based on Lubich's 3rd and 4th operators, this paper further provides new 4th order schemes for space fractional derivatives. The effectiveness of the new discretizations is verified theoretically and numerically.

\section*{Acknowledgments} This work was supported by the National Natural Science Foundation of China under
Grant No. 11271173, and the Program for New Century Excellent Talents in University under Grant No.
NCET-09-0438.

\end{document}